\documentclass[preprint,12pt]{elsarticle}

\usepackage{graphicx}
\usepackage{multirow}%
\usepackage{amsmath,amssymb,amsfonts}%
\usepackage{mathrsfs}%
\usepackage[title]{appendix}%
\usepackage{xcolor}%
\usepackage{textcomp}%
\usepackage{manyfoot}%
\usepackage{booktabs}%
\usepackage{algorithm}%
\usepackage{algorithmicx}%
\usepackage{algpseudocode}%
\usepackage{listings}%

\usepackage{tcolorbox}
\usepackage{tikz}

\usepackage{amsthm}%

\usepackage[colorlinks=true,citecolor=blue]{hyperref}
\newcommand{\redcite}[1]

\usepackage{etoolbox}
\makeatletter
    \patchcmd{\@author}{\global\let\@fnmark\@empty}{\global\let\@fnmark\@empty\global\let\@corref\@empty}{}{\@latex@error{Failed to patch \string\@author for \string\@corref reset}}
\makeatother

\usetikzlibrary{decorations.pathreplacing,calligraphy}

\newcommand{\grd}{{\sc Global Roman Domination}}

\newcommand{\ird}{{\sc Independent Roman Domination}}
\newcommand{\rthd}{{\sc Roman \{3\}-Domination}}
\newcommand{\rd}{{\sc Roman Domination}}
\newcommand{\resd}{{\sc Restrained Roman Domination}}
\newcommand{\urrd}{{\sc Unique Response Roman Domination}}
\newcommand{\prd}{{\sc Perfect Roman Domination}}
\newcommand{\ds}{{\sc Dominating Set}}
\newcommand{\xthc}{{\sc Exact 3-Cover}}
\newcommand{\xfc}{{\sc Exact 4-Cover}}
\newcommand{\xlc}{{\sc Exact $\ell$-Cover}}
\newcommand{\trtd}{{\sc Total Roman \{2\}-Domination}}

\usepackage{xcolor}
\definecolor{mynicegreen}{RGB}{70,150,70}

\newtheorem{theorem}{Theorem}[section]
\newtheorem{lemma}{Lemma}
{\bfseries}{\itshape}
{\bfseries}{\itshape}{\rmfamily}

\newtheorem{definition}{Definition}{\bfseries}{\itshape}{\rmfamily}

\journal{Theoretical Computer Science}

\begin{document}

\begin{frontmatter}



\title{On the complexity of global Roman domination problem in graphs}


\affiliation[label1]{organization={University of Hyderabad},
            city={Hyderabad},
            country={India}}
\affiliation[label2]{organization={Indian Institute of Technology Bhilai},
            city={Raipur},
            country={India}}
            
\author[label1]{Sangam Balchandar Reddy} \ead{21mcpc14@uohyd.ac.in}
\author[label1]{Arun Kumar Das} \ead{arunkumardas@uohyd.ac.in}
\author[label1]{Anjeneya Swami Kare} \ead{askcs@uohyd.ac.in}
\author[label2]{I. Vinod Reddy} \ead{vinod@iitbhilai.ac.in}

\begin{abstract}
A Roman dominating function of a graph $G=(V,E)$ is a labeling $f: V \rightarrow{} \{0 ,1, 2\}$  such that for each vertex $u \in V$ with $f(u) = 0$, there exists a vertex $v \in N(u)$ with $f(v) =2$. A Roman dominating function $f$ is a global Roman dominating function if it is a Roman dominating function for both $G$ and its complement $\overline{G}$. The weight of $f$ is the sum of $f(u)$ over all the vertices $u \in V$. The objective of \grd{} problem is to find a global Roman dominating function with minimum weight. 

In this paper, we study the algorithmic aspects of \grd{} problem on various graph classes and obtain the following results. 

\begin{enumerate}
    \item We prove that \rd{} and \grd{} problems are not computationally equivalent by identifying graph classes on which one is linear-time solvable, while the other is NP-complete.
    \item We show that \grd{} problem is NP-complete on split graphs, thereby resolving an open question posed by Panda and Goyal [Discrete Applied Mathematics, 2023].
    \item We prove that \grd{} problem is NP-complete on chordal bipartite graphs, planar bipartite graphs with maximum degree five and circle graphs.
    \item On the positive side, we present a linear-time algorithm for \grd{} problem on cographs.
\end{enumerate}
\end{abstract}




\end{frontmatter}



\section{Introduction}\rd{} (RD) is a fascinating concept in graph theory, inspired by the strategic placement of Roman legions to defend the Roman empire efficiently. The concept of RD problem was introduced by Stewart~\cite{stewart1999defend} in 1999. Given a graph $G = (V, E)$, a Roman dominating function $f: V \rightarrow{} \{0 ,1, 2\}$ is a labeling of vertices such that each vertex $u \in V$ with $f(u) = 0$ has a vertex $v \in N(u)$ with $f(v) =2$. The weight of a Roman dominating function $f$ is the sum of $f(u)$ over all the vertices $u \in V$. The minimum weight of a Roman dominating function for $G$ is denoted by $\gamma_{R}(G)$. The objective of RD problem is to compute $\gamma_{R}(G)$.

A Roman dominating function $f$ will be a global Roman dominating function if $f$ is a Roman dominating function for both $G$ and $\overline{G}$. The minimum weight of a global Roman dominating function $f$ is denoted by $\gamma_{gR}(G)$. The objective of \grd{} (GRD) problem is to compute $\gamma_{gR}(G)$.

The decision version of GRD problem is formally defined as follows.
\begin{tcolorbox}
{
\grd{}: \newline
\textbf{Input:} An instance $I$ = $(G, k)$, where $G=(V,E)$ is an undirected graph and an integer $k$.\newline
\textbf{Output:} Yes, if there exists a function $f: V \rightarrow{} \{0,1,2\}$ such that \\
(1) $\sum_{u \in V} f(u) \leq k$, \\
(2) for every vertex $u \in V$ with $f(u) = 0$, there exist vertices $v \in N_G(u)$ and $w \in N_{\overline{G}}(u)$ such that $f(v) = 2$ and $f(w) = 2$; No, otherwise.
}
\end{tcolorbox}
Liu et al.~\cite{liu2013roman} proved that RD problem is NP-complete on split graphs and bipartite graphs. They provided a linear-time algorithm for strongly chordal graphs. RD problem also admits linear-time algorithms for cographs, interval graphs, graphs of bounded clique-width and a polynomial-time algorithm for AT-free graphs~\cite{liedloff2008efficient}. Fernau~\cite{henning} studied RD problem in the realm of parameterized complexity and proved that the problem is W[2]-hard parameterized by weight. They provided an FPT algorithm that runs in time $\mathcal{O}^*(5^t)$, where $t$ is the treewidth of the input graph. In addition, for planar graphs, they showed that RD problem can be solved in $\mathcal{O}^*(3.3723^k)$ time. RD problem was proved to be W[1]-hard parameterized by weight, even for split graphs~\cite{mohanapriya2023roman}. Recently, Ashok et al.~\cite{ashok2024independent} provided an $\mathcal{O}^*(4^d)$ algorithm for the \ird{} (IRD) problem, where $d$ is the cluster vertex deletion set number. 

Several studies have investigated the complexity differences between various variants of \ds{} (DS) problem. Duan et al.~\cite{duan2022independent} proved that RD problem and IRD problem differ in complexity aspects. Chakradhar et al.~\cite{chakradhar2022algorithmic} showed that DS problem and \rthd{} problem are not equivalent in terms of complexity. Mann et al.~\cite{fernau2023perfect} demonstrated a complexity distinction between \urrd{} problem and \prd{} problem. It was proved that DS problem and \trtd{} problem differ in complexity~\cite{chakradhar2022roman2}. Chakradhar~\cite{chakradhar2023complexity} proved that DS problem and \resd{} problem are not equivalent in complexity. Along similar lines, we construct a graph class $\mathcal{F}$ for which RD problem can be solved in linear-time whereas GRD problem is NP-complete. We also construct a graph class $\mathcal{G}$ for which GRD problem can be solved in linear-time, whereas RD problem is NP-complete. Hence, we obtain that RD problem and GRD problem are not computationally equivalent.

Panda et al. \cite{panda} proved that GRD problem is NP-complete on bipartite graphs and chordal graphs. They also provided a polynomial-time algorithm for threshold graphs. In addition, they proposed an $\mathcal{O}(\ln{|V|})$-approximation algorithm and showed that there does not exist an $(1-\epsilon)\ln{|V|}$-approximation algorithm for any $\epsilon > 0$, unless P = NP. Moreover, they also proved that GRD problem is APX-complete for bounded-degree graphs. We advance the complexity study of the problem by proving that GRD problem is NP-complete on split graphs, thereby resolving an open question posed by Panda et al.~\cite{panda}. We also show that GRD problem is NP-complete on chordal bipartite graphs, planar bipartite graphs with maximum degree five and circle graphs. We further present a linear-time algorithm for GRD problem on cographs. As cographs form a superclass of threshold graphs, our result extends the known solution on threshold graphs.
\section{Preliminaries}
Let $G = (V, E)$ be a graph with $V$ as the vertex set and $E$ as the edge set such that $|V| = n$ and $|E| = m$. The sets $N(u) = \{v \in V~|~uv\in E\}$ and $N[u]= N(u) \cup \{u\}$ denote the open neighbourhood and closed neighbourhood of a vertex $u$ in $G$. We use $d(u)$ to denote the degree of a vertex $u$ such that $d(u) = |N(u)|$. A pendant vertex $u$ has degree, $d(u)$ = 1. A graph $G$ is $k$-regular, if $d(u) = k$ for each vertex $u \in V$. Two vertices $u$ and $v$ are \textit{false twins} in $G$ if and only if $N(u) = N(v)$ and $v \notin N(u)$. A \textit{universal vertex} in $G$ is a vertex that is adjacent to every other vertex in $G$. We use $\overline{G}$ to represent the complement graph of $G$. We say that a vertex $u \in V$ is \textit{covered} if either $f(u) \geq 1$ or there exist two vertices $v \in N(u)$ and $w \notin N(u)$ such that $f(v) = f(w) = 2$.  

For a vertex $u \in V$, the value of $f(u)$ is also called the \textbf{label} of $u$. For a set $T$, \textbf{weight} of $T$ is the sum of the labels of all the vertices in the set, i.e., $\sum\limits_{u \in T} f(u)$. We use $f(V)$ to denote the weight of a function $f$ on $G$, i.e., $f(V) = \sum\limits_{u \in V} f(u)$. In addition to this, we use the standard notations as defined in~\cite{WEST}.

\medskip
\noindent We define the \xlc{} problem as follows.
\begin{tcolorbox}
    \xlc: \newline
    \textbf{Input:} A finite set $X$ with $|X| = \ell q$ and a collection $C$ of $\ell$-element subsets of $X$. \newline
\textbf{Output:} Yes, if there exist $q$ pairwise disjoint subsets of $C$ whose union is $X$; No, otherwise.
\end{tcolorbox}
We establish the following result that will be used in Section \ref{sec4} of this paper.
\begin{lemma}~\label{x4cproof}
    \xfc{} problem is NP-complete.
\end{lemma}
\begin{proof}
    For an instance $(X = \{x_1,x_2,...x_{3q}\}, C = \{C_1,C_2,...C_t\})$ of X3C problem, we construct X4C problem instance $(X', C')$ as follows. We introduce an additional $q$ elements $x_{d_1}, x_{d_2},...,x_{d_q}$ in $X$ to obtain $X'$. Let $C_i^j$ be a 4-element set obtained by adding an element $x_{d_j}$ to $C_i$. We have $C'$ = \{$C_{1}^1, C_{1}^2 ,..., C_{1}^q, C_{2}^1,C_{2}^2 ..., C_{2}^q, ..., C_{t}^1,C_{t}^2 ..., C_{t}^q$\}. This concludes the construction of $(X', C')$. \\

\noindent ($\Rightarrow$) Let $C^* \subseteq C$ be an exact 3-cover of $X$ with $|C^*| = q$. We assign to each set $C_i \in C^*$ a unique dummy element $x_{d_j}$. For each such pair $(C_i, x_{d_j})$, we construct the corresponding 4-element set $C_i^j = C_i \cup \{x_{d_j}\} \in C'$.

Let $C'^* = \{C_i^j \mid C_i \in C^*,\ x_{d_j}\ \text{is uniquely assigned to}\ C_i\}$.

\begin{itemize}
    \item Each $C_i^j$ in $C'^*$ covers exactly 3 elements from $X$ and one unique element from $\{x_{d_1}, \dots, x_{d_q}\}$.
    \item The sets in $C'^*$ are disjoint, since the original $C_i$ sets are disjoint and the $x_{d_j}$ are distinct.
    \item The union of all sets in $C'^*$ is $X' = X \cup \{x_{d_1}, \dots, x_{d_q}\}$.
\end{itemize}

Thus, we have that $C'^*$ is an exact 4-cover of $X'$. \\

\noindent ($\Leftarrow$) Let $C'^* \subseteq C'$ be an exact 4-cover of $X'$, with $|C'^*| = q$. Each set in $C'$ is of the form $C_i^j = C_i \cup \{x_{d_j}\}$, so each selected set in $C'^*$ contains exactly one dummy element $x_{d_j}$ and three elements from $X$. There are exactly $q$ dummy elements in $X'$, and $q$ sets in $C'^*$, so each dummy element must appear in exactly one set in $C'^*$.

Let $C^* = \{C_i \mid C_i^j \in C'^*\}$ be the corresponding sets from the X3C problem instance.

\begin{itemize}
    \item $|C^*| = q$.
    \item Each $C_i \in C^*$ covers exactly 3 elements of $X$ and the sets are disjoint.
    \item The union of all sets in $C^*$ is $X$, since $X'$ is fully covered by $C'^*$ and all the dummy elements are uniquely used.
\end{itemize}

Thus, we conclude that $C^*$ is an exact 3-cover of $X$.
\end{proof}
\section{Complexity difference between \rd{} and \grd{}}~\label{sec4}
In this section, we show that the complexity of GRD problem differs from RD problem. We construct two graph classes $\mathcal{F}$ and $\mathcal{G}$ to demonstrate the complexity difference between these two problems. 
\subsection{Graph class: $\mathcal{F}$} In this subsection, we construct a graph class $\mathcal{F}$ for which RD problem can be solved in linear-time, whereas GRD problem is NP-complete. 

\medskip
\noindent \textbf{Construction of $H \in \mathcal{F}$.} Let $G = (V, E)$ be a 3-regular graph. We create three copies of $\overline{G}$ that are $G_1 = (V_1, E_1)$, $G_2 = (V_2, E_2)$ and $G_3 = (V_3, E_3)$. Vertex $u \in V(G)$ is denoted by $u_1$ in $G_1$, $u_2$ in $G_2$ and $u_3$ in $G_3$ respectively. For each vertex $u_1 \in V_1$, we make $u_1$ adjacent to every vertex in $N_{G_2}(u_2)$ and $N_{G_3}(u_3)$. For each vertex $u_2 \in V_2$, we make $u_2$ adjacent to every vertex in $N_{G_3}(u_3)$. We then introduce three vertices $v_1$, $v_2$ and $v_3$ and make each of them adjacent to every vertex in $V_1 \cup V_2 \cup V_3$ to obtain the resultant graph $H \in \mathcal{F}$. Formally $\mathcal{F}:=\{H~|~ \text{$H$ is obtained from a 3-regular graph $G$ by the above } \\ \text{procedure}\}$.
We show that RD problem is linear-time solvable on $\mathcal{F}$ whereas GRD problem on $\mathcal{F}$ is NP-complete.
\begin{lemma}~\label{f1}
    $\gamma_{R}(H) = 4$ for every $H \in \mathcal{F}$.
\end{lemma}
\begin{proof}
    Consider a graph $H \in \mathcal{F}$. One among the three vertices $v_1$, $v_2$ and $v_3$ is assigned the label 2 while the other two are assigned the label 1. All the other vertices are assigned the label 0 in $f$. Each vertex with label 0 is adjacent to $v_1$ (or $v_2$, $v_3$) with label 2. Thus, we have that $\gamma_{R}(H) \leq 4$. 
    
    In order to cover the vertices in $V_1 \cup V_2 \cup V_3$, one vertex among $\{v_1, v_2, v_3\}$ should be labeled 2. As $v_1$, $v_2$ and $v_3$ are non-adjacent, either both the vertices that are not assigned the label 2 from $\{v_1,v_2, v_3\}$ should be labeled 1 or at least one vertex in $V_1 \cup V_2 \cup V_3$ should be labeled 2. Thus, we have that $\gamma_{R}(H) \geq 4$. 
    
    Finally, we obtain that $\gamma_{R}(H) = 4$. 
\end{proof}
From Lemma \ref{f1}, we have the following result.
\begin{theorem}
    RD problem on graph class $\mathcal{F}$ is linear-time solvable.
\end{theorem}
Now, we show that GRD problem on  $\mathcal{F}$ is NP-complete following a reduction from DS problem on $3$-regular graphs. Given a graph $G = (V, E)$, a set $S \subseteq V$ is a dominating set if every vertex $u \in V\setminus S$ has a neighbour in $S$. The decision version of DS problem asks whether there exists a dominating set of size at most $k$.
The result related to DS problem on 3-regular graphs is given as follows.
\begin{theorem}[\cite{Kikuno}]~\label{ds_3r}
    DS problem on 3-regular graphs is NP-complete.
\end{theorem}

\medskip
\noindent \textbf{Construction.} Let $I = (G = (V, E), k)$ be an instance of DS problem, with $G$ being a $3$-regular graph. We construct an instance $I' = (G', k')$ of GRD problem as follows. We create three copies of $\overline{G}$, which are $G_1 = (V_1, E_1)$, $G_2 = (V_2, E_2)$ and $G_3 = (V_3, E_3)$. Vertex $u \in V(G)$ is denoted by $u_1$ in $G_1$, $u_2$ in $G_2$ and $u_3$ in $G_3$ respectively. For each vertex $u_1 \in V_1$, we make $u_1$ adjacent to every vertex in $N_{G_2}(u_2)$ and $N_{G_3}(u_3)$. For each vertex $u_2 \in V_2$, we make $u_2$ adjacent to every vertex in $N_{G_3}(u_3)$. We then introduce three vertices $v_1$, $v_2$ and $v_3$ and make each of them adjacent to every vertex in $V_1 \cup V_2 \cup V_3$ to obtain $G'$. It is easy to see that $G' \in \mathcal{F}$. We set $k' = 2k+2$. This concludes the construction of $I'$. See Figure~\ref{fig:fig1} for more details.
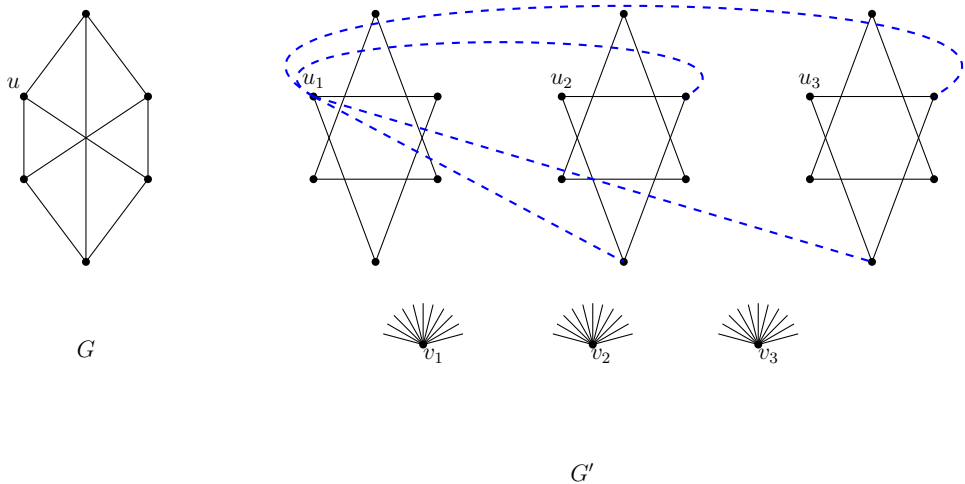
\begin{figure}
    \centering
     \begin{tikzpicture} [thick,scale=0.55, every node/.style={scale=0.75}]

        \filldraw (2, 1) circle (2pt) node[anchor=south]{};
        \filldraw (3.5, 3) circle (2pt) node[anchor=south]{};
        \filldraw (0.5, 3) circle (2pt) node[anchor=south]{};
        \filldraw (3.5, 5) circle (2pt) node[anchor=south]{};
        \filldraw (0.5, 5) circle (2pt) node[anchor=south]{};
        \filldraw (2, 7) circle (2pt) node[anchor=south]{};

        \draw[thin] (2, 1) -- (3.5, 3);
        \draw[thin] (2, 1) -- (0.5, 3);
        \draw[thin] (2, 1) -- (2, 7);
        \draw[thin] (3.5, 5) -- (3.5, 3);
        \draw[thin] (0.5, 5) -- (0.5, 3);
        \draw[thin] (3.5, 5) -- (2, 7);
        \draw[thin] (0.5, 5) -- (2, 7);
        \draw[thin] (3.5, 5) -- (0.5, 3);
        \draw[thin] (0.5, 5) -- (3.5, 3);
        
        \filldraw (0.25, 5) circle (0cm) node[anchor=south]{$u$};
        \filldraw (2, -1.5) circle (0cm) node[anchor=south]{$G$};

        \filldraw (9, 1) circle (2pt) node[anchor=south]{};
        \filldraw (10.5, 3) circle (2pt) node[anchor=south]{};
        \filldraw (7.5, 3) circle (2pt) node[anchor=south]{};
        \filldraw (10.5, 5) circle (2pt) node[anchor=south]{};
        \filldraw (7.5, 5) circle (2pt) node[anchor=south]{};
        \filldraw (9, 7) circle (2pt) node[anchor=south]{};

        \filldraw (15, 1) circle (2pt) node[anchor=south]{};
        \filldraw (16.5, 3) circle (2pt) node[anchor=south]{};
        \filldraw (13.5, 3) circle (2pt) node[anchor=south]{};
        \filldraw (16.5, 5) circle (2pt) node[anchor=south]{};
        \filldraw (13.5, 5) circle (2pt) node[anchor=south]{};
        \filldraw (15, 7) circle (2pt) node[anchor=south]{};

        \filldraw (21, 1) circle (2pt) node[anchor=south]{};
        \filldraw (22.5, 3) circle (2pt) node[anchor=south]{};
        \filldraw (19.5, 3) circle (2pt) node[anchor=south]{};
        \filldraw (22.5, 5) circle (2pt) node[anchor=south]{};
        \filldraw (19.5, 5) circle (2pt) node[anchor=south]{};
        \filldraw (21, 7) circle (2pt) node[anchor=south]{};

        \filldraw (7.5, 5) circle (0cm) node[anchor=south]{$u_1$};
        \filldraw (13.5, 5) circle (0cm) node[anchor=south]{$u_2$};
        \filldraw (19.5, 5) circle (0cm) node[anchor=south]{$u_3$};
        
        \filldraw (14, -4.5) circle (0cm) node[anchor=south]{$G'$};

        \draw[thin] (9, 1) -- (10.5, 5);
        \draw[thin] (9, 1) -- (7.5, 5);
        \draw[thin] (9, 7) -- (10.5, 3);
        \draw[thin] (9, 7) -- (7.5, 3);
        \draw[thin] (7.5, 3) -- (10.5, 3);
        \draw[thin] (10.5, 5) -- (7.5, 5);

        \draw[thin] (15, 1) -- (16.5, 5);
        \draw[thin] (15, 1) -- (13.5, 5);
        \draw[thin] (15, 7) -- (16.5, 3);
        \draw[thin] (15, 7) -- (13.5, 3);
        \draw[thin] (13.5, 3) -- (16.5, 3);
        \draw[thin] (16.5, 5) -- (13.5, 5);

        \draw[thin] (21, 1) -- (22.5, 5);
        \draw[thin] (21, 1) -- (19.5, 5);
        \draw[thin] (21, 7) -- (22.5, 3);
        \draw[thin] (21, 7) -- (19.5, 3);
        \draw[thin] (19.5, 3) -- (22.5, 3);
        \draw[thin] (22.5, 5) -- (19.5, 5);

        \draw[thin] (10.15, -1) -- (10.15,0);\draw[thin] (10.15, -1) -- (11.115,-0.741);
        \draw[thin] (10.15, -1) -- (11.016,-0.5);\draw[thin] (10.15, -1) -- (10.857,-0.293);
        \draw[thin] (10.15, -1) -- (10.65,-0.133);\draw[thin] (10.15, -1) -- (10.391,-0.034);

        \draw[thin] (10.15, -1) -- (9.909, -0.034);
        \draw[thin] (10.15, -1) -- (9.65, -0.133);\draw[thin] (10.15, -1) -- (9.443, -0.293);
        \draw[thin] (10.15, -1) -- (9.284, -0.5);\draw[thin] (10.15, -1) -- (9.185, -0.741);
        
        \draw[thick, blue, dashed] (7.5, 5) -- (15, 1);
        \draw[thick, blue, dashed] (7.5, 5) to [out=150,in=30] (16.5, 5);

        \draw[thick, blue, dashed] (7.5, 5) -- (21, 1);
        \draw[thick, blue, dashed] (7.5, 5) to [out=150,in=30] (22.5, 5);

        \filldraw (10.15, -1) circle (2pt) node[anchor=south]{};
        \filldraw (10.4, -1.65) circle (0cm) node[anchor=south]{$v_1$};

        \draw[thin] (14.25, -1) -- (14.25,0);
        \draw[thin] (14.25, -1) -- (15.215,-0.741);
        \draw[thin] (14.25, -1) -- (15.116,-0.5);\draw[thin] (14.25, -1) -- (14.957,-0.293);
        \draw[thin] (14.25, -1) -- (14.75,-0.134);\draw[thin] (14.25, -1) -- (14.491,-0.034);

        \draw[thin] (14.25, -1) -- (14.009, -0.034);
        \draw[thin] (14.25, -1) -- (13.75, -0.134);\draw[thin] (14.25, -1) -- (13.543, -0.293);
        \draw[thin] (14.25, -1) -- (13.384, -0.5);\draw[thin] (14.25, -1) -- (13.285, -0.741);

        \filldraw (18.25, -1) circle (2pt) node[anchor=south]{};

        \draw[thin] (18.25, -1) -- (18.25,0);
        \draw[thin] (18.25, -1) -- (19.215,-0.741);
        \draw[thin] (18.25, -1) -- (19.116,-0.5);\draw[thin] (18.25, -1) -- (18.957,-0.293);
        \draw[thin] (18.25, -1) -- (18.75,-0.134);\draw[thin] (18.25, -1) -- (18.491,-0.034);

        \draw[thin] (18.25, -1) -- (18.009, -0.034);
        \draw[thin] (18.25, -1) -- (17.75, -0.134);\draw[thin] (18.25, -1) -- (17.543, -0.293);
        \draw[thin] (18.25, -1) -- (17.384, -0.5);\draw[thin] (18.25, -1) -- (17.285, -0.741);

        \filldraw (14.25, -1) circle (2pt) node[anchor=south]{};
        \filldraw (14.45, -1.65) circle (0cm) node[anchor=south]{$v_2$};
        \filldraw (18.5, -1.65) circle (0cm) node[anchor=south]{$v_3$};
        
    \end{tikzpicture} \vspace{2mm}
        \caption{(left) A $3$-regular graph $G$; (right) The construction of  $G' \in \mathcal{F}$ from $G$. Only the adjacency of a vertex $u_1 \in G_1$ to its neighbors in $G_2$ and $G_3$ is shown in $G'$.}
    \label{fig:fig1}
\end{figure}
\begin{lemma}~\label{one2}
    Let $f$ be a minimum weighted global Roman dominating function for $G'$. One vertex among $v_1$,  $v_2$ and $v_3$ is assigned the label 2 in $f$ while the other two are assigned the label 0.
\end{lemma}
\begin{proof}
        Each vertex among $v_1$, $v_2$ and $v_3$ is adjacent to every vertex in $V_1 \cup V_2 \cup V_3$ and its only non-neighbours in $G'$ are the other two vertices from $v_1$, $v_2$ and $v_3$. There are two possible ways to assign labels to $v_1$, $v_2$ and $v_3$ in $f$:
    \begin{itemize}
        \item label 2 to one of them and 0 to the other two, or
        \item label 1 to all three.
    \end{itemize}

In the second case, assigning label \(1\) to each of \(v_1\), \(v_2\) and \(v_3\) yields a total weight of three but do not cover any vertex in $V_1 \cup V_2 \cup V_3$.  
In contrast, in the first case, assigning label \(2\) to one of the vertices and label \(0\) to the other two yields a total weight of two and also covers all the vertices in $V_1 \cup V_2 \cup V_3$. Therefore, the labeling in which one among \(v_1\), \(v_2\) and $v_3$ receives label \(2\) and the other two receives label \(0\) is optimal. Thus, we have the label 2 to one among $v_1$, $v_2$ and $v_3$ while the other two are assigned the label 0 in $f$. We assume that $f(v_1) = 2$, $f(v_2) = 0$ and $f(v_3) = 0$.

To serve as a non-neighbour with label 2 to the vertices in $V_1 \cup V_2 \cup V_3$, at least one vertex in $V_1 \cup V_2 \cup V_3$ should be labeled 2. Vertices $v_2$ and $v_3$ are covered, as they have a neighbour in $V_1 \cup V_2 \cup V_3$ with label 2 and a non-neighbour $v_1$ with label 2. Hence, all the vertices in $G'$ are covered and at most one among $v_1$,  $v_2$ and $v_3$ is assigned the label 2 in $f$ while the other two are assigned the label 0.
\end{proof}
Without loss of generality, for the rest of this subsection, we assume that $f(v_1) =2$, $f(v_2) = 0$ and $f(v_3) = 0$.
\begin{lemma}~\label{f2}
    Let $f$ be a minimum weighted global Roman dominating function for $G'$. Each vertex in $G'$ is assigned the labels from $\{0, 2\}$ in $f$.
\end{lemma}
\begin{proof}
    Each vertex $w \in V_1 \cup V_2 \cup V_3$ with $f(w) = 0$ has a neighbour $v_1$ with $f(v_1) = 2$. Therefore, we only look for a non-neighbour to $w$ with label 2. Consider the vertices $u_1$, $u_2$ and $u_3$ that are the copies of vertex $u$ in $G_1$, $G_2$ and $G_3$ respectively. As $u_1$, $u_2$ and $u_3$ are false twins, at most one of them is assigned the label 2 in $f$. In order for the vertices $u_1$, $u_2$ and $u_3$ to be covered, either all the three vertices $u_1$, $u_2$ and $u_3$ should be assigned the label at least 1 or one vertex from the set $(V'\setminus N_{G'}(u_1))$ should be assigned the label 2 in $f$. Let us assume that vertex $u_1$ is assigned the label 1, then $u_2$ and $u_3$ must be assigned the label 1 as well, for $f$ to be minimum weighted. Instead, we can label one vertex from the set $(V'\setminus N_{G'}(u_1))$ with 2 to cover all three vertices $u_1$, $u_2$ and $u_3$, which is an optimal labeling compared to the labeling of 1 for three vertices $u_1$, $u_2$ and $u_3$. Therefore, we conclude that every $f$ assigns the labels from $\{0,2\}$ to the vertices in $G'$. 
\end{proof}
\begin{lemma}~\label{f3}
    $G$ has a dominating set of size at most $k$ if and only if $G'$ has a global Roman dominating function of weight at most $2k+2$.
\end{lemma}
\begin{proof}
\noindent($\Rightarrow$) Let $S$ be a dominating set for $G$ of size at most $k$. We obtain a global Roman dominating function $f$ for $G'$ as follows.

For each vertex $u \in S$, we assign label 2 to $u_1$ and label 0 to both $u_2$ and $u_3$ in $f$. All the other vertices in $G'$, excluding $v_1$, $v_2$ and $v_3$ are assigned label 0. Finally, vertex $v_1$ receives label 2 while both $v_2$ and $v_3$ are assigned label 0. 

Each vertex $u_i \in V_1 \cup V_2 \cup V_3$ with $f(u_i) = 0$ has a neighbour $v_1 \in V'$ with $f(v_1) = 2$ and also has a non-neighbour $x_1$ with $f(x_1) = 2$ where $x \in N_G(u)$. Vertices $v_2$ and $v_3$ have a neighbour $x_1 \in V_1 \cup V_2 \cup V_3$ with $f(x_1)=2$ for some $x \in V(G)$ and also has a non-neighbour $v_1$ with $f(v_1)=2$. All the other vertices in $G'$ are covered, as they have the label greater than 0. Thus, we conclude that $f$ is a global Roman dominating function of weight at most $2k+2$.

\noindent($\Leftarrow$) Let $f$ be a global Roman dominating function of weight at most $2k+2$. We obtain a dominating set $S$ as follows.

From Lemma \ref{one2}, we have that $f(v_1)=2$, $f(v_2)=0$ and $f(v_3) = 0$. At this point, we have utilized a weight of two for the set $\{v_1, v_2, v_3\}$ and a weight of $2k$ is available for the vertices in $V_1 \cup V_2 \cup V_3$.

Each vertex $u_i \in V_1 \cup V_2 \cup V_3$ with $f(u_i) = 0$ has a neighbour $v_1 \in V'$ with $f(v_1) = 2$. Due to Lemma \ref{f2}, we assume that no vertex from $V'$ is assigned the label 1 in $f$. Each vertex from $V'$ with label 0 must have a non-neighbour with label 2. We assign label 2 to at most $k$ vertices in $G' \setminus \{v_1, v_2, v_3\}$ and the rest of the vertices are assigned the label 0. The non-neighbour to $u_i$ in $G'$ that is assigned the label 2 is a neighbour to $u$ in $G$ that is in $S$. The $k$ vertices in $G' \setminus \{v_1, v_2, v_3\}$ that are labeled 2 in $f$ will correspond to the dominating set $S$ of size at most $k$ in $G$. 
\end{proof} 
As $G' \in \mathcal{F}$, with the help of Theorem \ref{ds_3r} and Lemma \ref{f3}, we arrive at the following theorem. 
\begin{theorem}
    GRD problem on graph class $\mathcal{F}$ is NP-complete.
\end{theorem}
\subsection{Graph class: $\mathcal{G}$}
In this subsection, we construct a graph class $\mathcal{G}$ for which GRD problem can be solved in linear-time whereas RD problem is NP-complete. 

\begin{figure}[t]
    \centering
     \begin{tikzpicture} [thick,scale=0.75, every node/.style={scale=1}]
        \draw (4, 3.875) ellipse (1.25 and 4.75);
        \draw (9, 3.875) ellipse (1.25 and 4.75);

        \filldraw (4, 6.5) circle (2pt) node[anchor=south]{};
        \filldraw (4, 4.75) circle (2pt) node[anchor=south]{};
        \filldraw (4, 3.25) circle (2pt) node[anchor=south]{};
        \filldraw (4, 1.5) circle (2pt) node[anchor=south]{};

        \filldraw (9, 0.5) circle (2pt) node[anchor=south]{};
        \filldraw (9, 7.5) circle (2pt) node[anchor=south]{};
        \filldraw (9, 6.5) circle (2pt) node[anchor=south]{};
        \filldraw (9, 5.5) circle (2pt) node[anchor=south]{};
        \filldraw (9, 4.5) circle (2pt) node[anchor=south]{};
        \filldraw (9, 3.5) circle (2pt) node[anchor=south]{};
        \filldraw (9, 2.5) circle (2pt) node[anchor=south]{};
        \filldraw (9, 1.5) circle (2pt) node[anchor=south]{};

        \filldraw (12, 6.5) circle (2pt) node[anchor=south]{};
        \filldraw (12, 5.5) circle (2pt) node[anchor=south]{};
        \filldraw (12, 4.5) circle (2pt) node[anchor=south]{};
        \filldraw (12, 3.5) circle (2pt) node[anchor=south]{};
        \filldraw (12, 2.5) circle (2pt) node[anchor=south]{};
        \filldraw (12, 1.5) circle (2pt) node[anchor=south]{};
        \filldraw (12, 0.5) circle (2pt) node[anchor=south]{};
        \filldraw (12, 7.5) circle (2pt) node[anchor=south]{};

        \filldraw (14, 7.8) circle (1pt) node[anchor=south]{};
        \filldraw (14, 6.8) circle (1pt) node[anchor=south]{};
        \filldraw (14, 7.2) circle (1pt) node[anchor=south]{};
        \filldraw (14, 5.8) circle (1pt) node[anchor=south]{};
        \filldraw (14, 6.2) circle (1pt) node[anchor=south]{};
        \filldraw (14, 4.8) circle (1pt) node[anchor=south]{};
        \filldraw (14, 5.2) circle (1pt) node[anchor=south]{};
        \filldraw (14, 3.8) circle (1pt) node[anchor=south]{};
        \filldraw (14, 4.2) circle (1pt) node[anchor=south]{};
        \filldraw (14, 2.8) circle (1pt) node[anchor=south]{};
        \filldraw (14, 3.2) circle (1pt) node[anchor=south]{};
        \filldraw (14, 1.8) circle (1pt) node[anchor=south]{};
        \filldraw (14, 2.2) circle (1pt) node[anchor=south]{};
        \filldraw (14, 0.8) circle (1pt) node[anchor=south]{};
        \filldraw (14, 1.2) circle (1pt) node[anchor=south]{};
        \filldraw (14, 0.2) circle (1pt) node[anchor=south]{};

        \filldraw (16, 7.5) circle (2pt) node[anchor=south]{};
        \filldraw (16, 5.5) circle (2pt) node[anchor=south]{};
        \filldraw (16, 3.5) circle (2pt) node[anchor=south]{};
        \filldraw (16, 1.5) circle (2pt) node[anchor=south]{};

        \filldraw (1.25, 4) circle (2pt) node[anchor=south]{};

        \filldraw (1.25, 3) circle (2pt) node[anchor=south]{};
        \filldraw (0, 4) circle (2pt) node[anchor=south]{};
        \filldraw (1.25, 5) circle (2pt) node[anchor=south]{};

        \filldraw (14.85, 6.7) circle (1pt) node[anchor=south]{};
        \filldraw (14.85, 6.5) circle (1pt) node[anchor=south]{};
        \filldraw (14.85, 6.3) circle (1pt) node[anchor=south]{};
        
        \draw[thin, gray] (14.85, 6.7) -- (14, 6.8);
        \draw[thin, gray] (14.85, 6.5) -- (14, 6.8);
        \draw[thin, gray] (14.85, 6.3) -- (14, 6.8);       
        
        \draw[thin, gray] (14.85, 6.7) -- (14, 6.2);
        \draw[thin, gray] (14.85, 6.5) -- (14, 6.2);
        \draw[thin, gray] (14.85, 6.3) -- (14, 6.2);         
        
        \draw[thin, gray] (14.85, 6.7) -- (12, 6.5);
        \draw[thin, gray] (14.85, 6.5) -- (12, 6.5);
        \draw[thin, gray] (14.85, 6.3) -- (12, 6.5);    
        
        \draw[thin] (0, 4) -- (1.25, 4);
        \draw[thin] (0, 4) -- (1.25, 3);
        \draw[thin] (1.25, 5) -- (0, 4);

        \draw[thin] (1.25, 4) -- (4, 6.5);
        \draw[thin] (1.25, 4) -- (4, 4.75);
        \draw[thin] (1.25, 4) -- (4, 3.25);
        \draw[thin] (1.25, 4) -- (4, 1.5);

        \draw[thin] (4, 6.5) -- (9, 7.5);
        \draw[thin] (4, 6.5) -- (9, 6.5);
        \draw[thin] (4, 6.5) -- (9, 5.5);
        \draw[thin] (4, 6.5) -- (9, 4.5);
        \draw[thin] (4, 4.75) -- (9, 6.5);
        \draw[thin] (4, 4.75) -- (9, 4.5);
        \draw[thin] (4, 4.75) -- (9, 2.5);
        \draw[thin] (4, 4.75) -- (9, 1.5);
        \draw[thin] (4, 3.25) -- (9, 5.5);
        \draw[thin] (4, 3.25) -- (9, 3.5);
        \draw[thin] (4, 3.25) -- (9, 2.5);
        \draw[thin] (4, 3.25) -- (9, 1.5);
        \draw[thin] (4, 1.5) -- (9, 3.5);
        \draw[thin] (4, 1.5) -- (9, 2.5);
        \draw[thin] (4, 1.5) -- (9, 1.5);
        \draw[thin] (4, 1.5) -- (9, 0.5);

        \draw[thin] (9, 0.5) -- (12, 0.5);
        \draw[thin] (9, 1.5) -- (12, 1.5);
        \draw[thin] (9, 2.5) -- (12, 2.5);
        \draw[thin] (9, 3.5) -- (12, 3.5);
        \draw[thin] (9, 4.5) -- (12, 4.5);
        \draw[thin] (9, 5.5) -- (12, 5.5);
        \draw[thin] (9, 6.5) -- (12, 6.5);
        \draw[thin] (9, 7.5) -- (12, 7.5);

        \draw[thin] (14, 0.2) -- (12, 0.5);
        \draw[thin] (14, 0.8) -- (12, 0.5);
        \draw[thin] (14, 1.2) -- (12, 1.5);
        \draw[thin] (14, 1.8) -- (12, 1.5);
        \draw[thin] (14, 2.2) -- (12, 2.5);
        \draw[thin] (14, 2.8) -- (12, 2.5);
        \draw[thin] (14, 3.2) -- (12, 3.5);
        \draw[thin] (14, 3.8) -- (12, 3.5);
        \draw[thin] (14, 4.2) -- (12, 4.5);
        \draw[thin] (14, 4.8) -- (12, 4.5);
        \draw[thin] (14, 5.2) -- (12, 5.5);
        \draw[thin] (14, 5.8) -- (12, 5.5);
        \draw[thin] (14, 6.2) -- (12, 6.5);
        \draw[thin] (14, 6.8) -- (12, 6.5);
        \draw[thin] (14, 7.2) -- (12, 7.5);
        \draw[thin] (14, 7.8) -- (12, 7.5);

        \draw[thin] (14, 1.2) -- (16, 1.5);
        \draw[thin] (14, 1.8) -- (16, 1.5);
        \draw[thin] (14, 3.2) -- (16, 3.5);
        \draw[thin] (14, 3.8) -- (16, 3.5);
        \draw[thin] (14, 5.2) -- (16, 5.5);
        \draw[thin] (14, 5.8) -- (16, 5.5);
        \draw[thin] (14, 7.2) -- (16, 7.5);
        \draw[thin] (14, 7.8) -- (16, 7.5);

        \draw[thin] (14, 0.8) -- (14, 1.2);
        \draw[thin] (14, 1.8) -- (14, 2.2);
        \draw[thin] (14, 2.8) -- (14, 3.2);
        \draw[thin] (14, 3.8) -- (14, 4.2);
        \draw[thin] (14, 4.8) -- (14, 5.2);
        \draw[thin] (14, 5.8) -- (14, 6.2);
        \draw[thin] (14, 6.8) -- (14, 7.2);
        \draw[thin] (14, 7.8) to [out=90,in=90] (18, 6);
        \draw[thin] (18, 2) -- (18, 6);
        \draw[thin] (18, 2) to [out=270,in=270] (14, 0.2);

        \draw[thin] (1.25, 5) -- (4, 6.5);
        \draw[thin] (1.25, 5) -- (4, 4.75);
        \draw[thin] (1.25, 5) -- (4, 3.25);
        \draw[thin] (1.25, 5) -- (4, 1.5);        
        \draw[thin] (1.25, 3) -- (4, 6.5);
        \draw[thin] (1.25, 3) -- (4, 4.75);
        \draw[thin] (1.25, 3) -- (4, 3.25);
        \draw[thin] (1.25, 3) -- (4, 1.5);

        \draw[thick, dotted] (11.5, 8) -- (12.5, 8);
        \draw[thick, dotted] (11.5, 0) -- (11.5, 8);
        \draw[thick, dotted] (12.5, 0) -- (12.5, 8);
        \draw[thick, dotted] (11.5, 0) -- (12.5, 0);

        \draw[thick, dotted] (13.5, 8) -- (14.5, 8);
        \draw[thick, dotted] (13.5, 0) -- (13.5, 8);
        \draw[thick, dotted] (14.5, 0) -- (14.5, 8);
        \draw[thick, dotted] (13.5, 0) -- (14.5, 0);

        \draw[thick, dotted] (15.5, 7.75) -- (16.5, 7.75);
        \draw[thick, dotted] (15.5, 0.25) -- (15.5, 7.75);
        \draw[thick, dotted] (16.5, 0.25) -- (16.5, 7.75);
        \draw[thick, dotted] (15.5, 0.25) -- (16.5, 0.25);
        
        \draw[thick] (1.25, 5) to [out=120,in=180] (7, 10); 
        \draw[thick] (7, 10) to [out=0,in=150] (12, 8);
        \draw[thick] (7, 10) to [out=0,in=150] (13.8, 8);
        \draw[thick] (1.25, 3) to [out=-120,in=180] (7, -2);
        \draw[thick] (7, -2) to [out=0,in=210] (12, 0);
        \draw[thick] (7, -2) to [out=0,in=210] (13.8, 0);

        \filldraw (5, -1) circle (0cm) node[anchor=south]{$A$};
        \filldraw (7.75, -1) circle (0cm) node[anchor=south]{$B$};
        \filldraw (11, -0.25) circle (0cm) node[anchor=south]{$P$};
        \filldraw (15, -0.25) circle (0cm) node[anchor=south]{$Q$};
        \filldraw (17, 0.5) circle (0cm) node[anchor=south]{$R$};

        \filldraw (15.2, 6.5) circle (0cm) node[anchor=south]{\tiny $s_2^1$};
        \filldraw (15.2, 5.9) circle (0cm) node[anchor=south]{\tiny $s_2^3$};

        \filldraw (9.25, 7.5) circle (0cm) node[anchor=south]{$b_1$};
        \filldraw (9.25, 6.5) circle (0cm) node[anchor=south]{$b_2$};
        \filldraw (9.25, 5.5) circle (0cm) node[anchor=south]{$b_3$};
        \filldraw (9.25, 4.5) circle (0cm) node[anchor=south]{$b_4$};
        \filldraw (9.25, 3.5) circle (0cm) node[anchor=south]{$b_5$};
        \filldraw (9.25, 2.5) circle (0cm) node[anchor=south]{$b_6$};
        \filldraw (9.25, 1.5) circle (0cm) node[anchor=south]{$b_7$};
        \filldraw (9.25, 0.5) circle (0cm) node[anchor=south]{$b_8$};

        \filldraw (4.25, 6.5) circle (0cm) node[anchor=south]{$a_1$};
        \filldraw (4.25, 4.85) circle (0cm) node[anchor=south]{$a_2$};
        \filldraw (4.25, 2.5) circle (0cm) node[anchor=south]{$a_3$};
        \filldraw (4.25, 0.75) circle (0cm) node[anchor=south]{$a_4$};
        
        \filldraw (-0.25, 4) circle (0cm) node[anchor=south]{$u$};
        \filldraw (0.9, 2.8) circle (0cm) node[anchor=south]{$x$};
        \filldraw (0.9, 4) circle (0cm) node[anchor=south]{$w$};
        \filldraw (0.9, 4.8) circle (0cm) node[anchor=south]{$v$};
        
    \end{tikzpicture}
        \caption{Graph $G' \in \mathcal{G}$ constructed from an X4C problem instance: $X = \{1,2,3,4,5,6, 7, 8\}$ and $C$ = $\{C_1 =  \{1,2,3,4\}$, $C_2 =  \{2,4,6,7\}$, $C_3 =  \{3,5,6,7\}$ and $C_4 =  \{5,6,7,8\}\}$. A thick edge between a vertex $v$ and a set $Y$ indicates that $v$ is adjacent to every vertex of set $Y$. The edges between the vertices of set $A$ are not shown in the figure. All the vertices of the set $S$ are not shown in the figure due to space constraints. }
    \label{fig:fig2}
\end{figure}
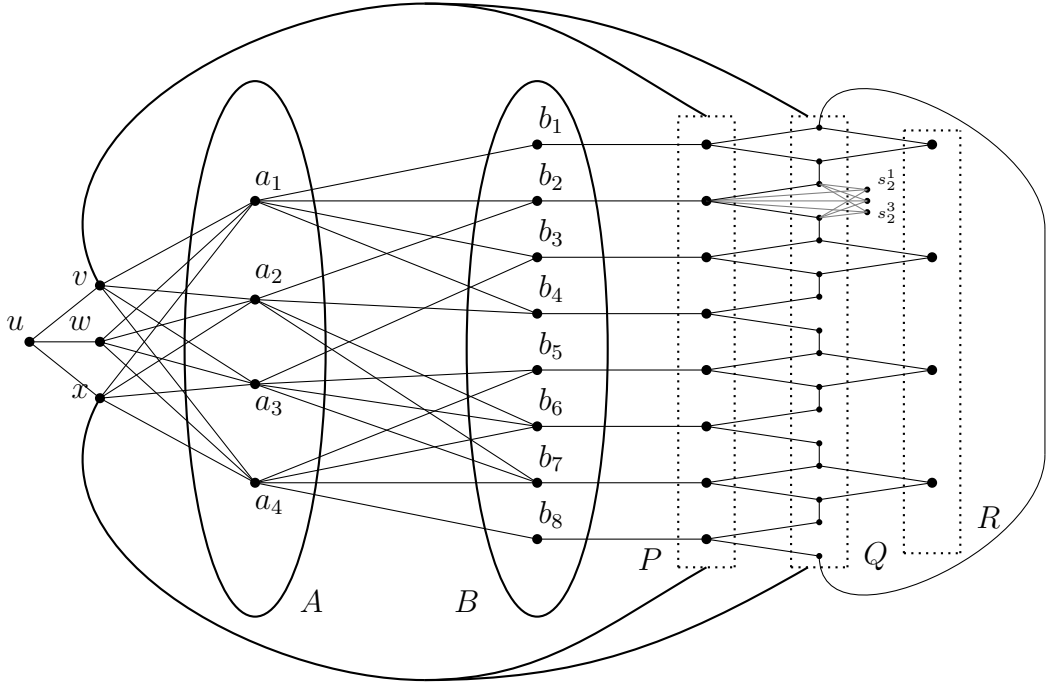 
\medskip
\noindent \textbf{Construction of $H \in \mathcal{G}$. }We consider an instance of \xfc{} (X4C) problem with $|X| = 4l$. For each element $x_i$ in the X4C problem instance, we create a vertex $b_i$ in $\mathcal{G}$. Similarly, for each set $C_j$ in the X4C problem instance, we create a vertex $a_j$ in $\mathcal{G}$. 

Let $A = \bigcup\limits_{j \in [m]} a_j$ and $B = \bigcup\limits_{i \in [4l]} b_i$. We construct a split graph where $A$ forms a clique of size $m$ and $B$ forms an independent set of size $4l$.
For every $i\in [4l]$ and $j \in [m]$, we add an edge between $b_i\in B$ and $a_j \in A$ if and only if $x_i \in C_j$ in the given X4C problem instance.


We have the following gadget adjacent to $A$. We create three vertices $v, w$ and $x$ and make each of them adjacent to every vertex in $A$. We also create a vertex $u$, which is adjacent to $v$, $w$ and $x$.

We have the following gadget adjacent to $B$. For each $i \in [4l]$, we create a vertex set $\{p_i$, $q_i^1$, $q_i^2$, $s_i^1$, $s_i^2$, $s_i^3\}$. We make $b_i$ adjacent to $p_i$. Vertex $p_i$ will be made adjacent to both $q_i^1$ and $q_i^2$. Vertices of the set $\{s_i^1, s_i^2, s_i^3\}$ are made adjacent to each vertex among $p_i$, $q_i^1$ and $q_i^2$. For $i \in [4l]$ and $i \pmod 2 \equiv 1$, we create a vertex $r_i$ and make it adjacent to $q_i^1$ and $q_i^2$. We also make $q_i^2$ adjacent to $q_{i+1}^1$ for each $i \in [4l-1]$. We make the vertices $q_{4l}^2$ and $q_1^1$ adjacent. Let $P = \bigcup\limits_{i \in [4l]} p_i$, $Q = \bigcup\limits_{i \in [4l]} \{q_i^1, q_i^2\}$, $R = \bigcup\limits_{i \in \{1,3,...,4l-1\}} r_i$ and $S = \bigcup\limits_{i \in [4l]} \{s_i^1, s_i^2, s_i^3\}$.

We also make both the vertices $v$ and $x$ adjacent to every vertex in $P \cup Q $.

This concludes the construction of graph $H \in \mathcal{G}$ based on the given X4C problem instance. See Figure~\ref{fig:fig2} for an illustration.

\medskip\noindent
\textbf{Idea.} We show that for graph class $\mathcal{G}$, there exists a Roman dominating function of weight $10l+1$ if and only if the underlying X4C problem is a YES-instance. But, from Lemma \ref{x4cproof}, we have that X4C problem is NP-complete. Hence, we obtain that RD problem on graph class $\mathcal{G}$ is NP-complete. In case of GRD problem, we show that the minimum weight of global Roman dominating function is fixed to $10l+2$ and does not depend on the underlying X4C problem instance. Hence, we have that GRD problem on graph class $\mathcal{G}$ is linear-time solvable.
\begin{lemma}~\label{weight1}
     The weight of a Roman dominating function $f$ on $\{u,v,w,x\}$ is at least $1$.
\end{lemma}
\begin{proof}
As vertex $u$ is adjacent only to $v,w$ and $x$, either $u$ should be assigned the label 1 or at least one among $v,w$ and $x$ should be assigned the label 2. Thus, the weight of $f$ on $\{u,v,w,x\}$ is at least $1$.
\end{proof}
\begin{lemma}~\label{weight8l}
    The weight of a Roman dominating function $f$ on $P \cup Q \cup S$ is at least $8l$.
\end{lemma}
\begin{proof}
    For each $i \in [4l]$: vertices $p_i, q_i^1$ and $q_i^2$  are adjacent to three distinct vertices $s_i^1, s_i^2$ and $s_i^3$. Consider the set $\{p_i, q_i^1, q_i^2, s_i^1, s_i^2, s_i^3\}$. We either label at least one vertex among $p_i, q_i^1$ and $q_i^2$ with 2 or each vertex among $\{s_i^1, s_i^2, s_i^3\}$ with at least 1. Thus, the weight of $\{p_i, q_i^1, q_i^2, s_i^1, s_i^2, s_i^3\}$ is at least two. Hence, the weight of $P \cup Q \cup S$ in $f$ is at least $8l$. 
\end{proof}

\begin{lemma}~\label{weight2b}
     The weight of a Roman dominating function $f$ on $A \cup B \cup P$ is at least $2l$.
\end{lemma}
\begin{proof}
    In order to cover the vertices of $B$, either $l$ vertices from $A$ can be labeled 2 or a subset of vertices from $B \cup P$ can have non-zero labels. A function that assigns the label 2 to $l$ vertices in $A$ has a weight of $2l$ on $A \cup B \cup P$. Whereas, a function that assigns non-zero labels to a subset of vertices from $B \cup P$ leads to a larger weight than $2l$. This is because a vertex from $A$ covers four vertices of $B$, which costs a weight of two. But, if these four vertices of $B$ were to have the label 1 or to be covered by their neighbours in $P$, the weight needed is four. Hence, we obtain that the weight of $f$ on $A \cup B \cup P$ is at least $2l$.
\end{proof}
\begin{lemma}~\label{sameindex}
    Consider a Roman dominating function $f$ such that for a vertex $y \in \bigcup\limits_{i \in [4l]} q_i^1$, $f(y) = 2$. The function $f$ admits a minimum weight if and only if for each vertex $z \in \bigcup\limits_{i \in [4l]} q_i^1$, $f(z) = 2.$
\end{lemma}
\begin{proof}
    From Lemma \ref{weight2b}, we have that a minimum weighted Roman dominating function $f$ assigns the label 2 to the vertices in $\bigcup\limits_{C_j \in C'} a_j$, where $C'$ is an exact cover of the X4C problem instance. 
    
    Let $y \in \bigcup\limits_{i \in [4l]} q_i^1$ such that $f(y) = 2$. For the sake of contradiction, we assume that $f(q_i^1) \leq 1$ for some $i \in [4l]$. In order to cover the vertices in $\{p_i,q_i^2, s_i^1, s_i^2, s_i^3\}$, we must label $p_i$ or $q_i^2$ with 2, but $q_{i-1}^2$ remains uncovered, thereby increasing the weight of $f$ by at least one. We have the following three cases involving the four vertices in $\bigcup\limits_{i \in [4l]} q_i^1$ with the label smaller than 2. The reason for considering four vertices from $\bigcup\limits_{i \in [4l]} q_i^1$ is because we are working on the X4C problem instance and the four vertices in $\bigcup\limits_{i \in [4l]} b_i^1$ can be covered by exactly one vertex from $\bigcup\limits_{j \in [m]} a_j$.
    
    \medskip\noindent
    \textbf{Case 1: }We consider four vertices $q_{t+1}^1,q_{t+2}^1,q_{t+3}^1$ and $q_{t+4}^1$ with consecutive indices such that $b_{t+1}$, $b_{t+2}$, $b_{t+3}$ and $b_{t+4}$ can be covered by a vertex in $A$ and label them with 0. For each vertex $q_i^1$ with the label 0, one among $q_i^2$ and $p_i$ must be assigned the label 2. Thus, we have the following subcases.

    \medskip\noindent
    \textbf{Case 1a:} We label the vertices $p_{t+1},p_{t+2},p_{t+3}$ and $p_{t+4}$ with 2. As there exists a vertex from $A$ (say $a_j$) that covers $b_{t+1},b_{t+2},b_{t+3}$ and $b_{t+4}$, we can label $a_j$ with 0. As the vertices of $R$ need to be covered, we label two vertices of $R$ with 1. This takes up a weight of two that we gained from $a_j$. But, we still need an additional weight of one to cover $q_{t-1}^2$. Hence, this leads to a larger weight.

    \medskip\noindent
    \textbf{Case 1b:} We label the vertices $q_{t+1}^2,q_{t+2}^2,q_{t+3}^2$ and $q_{t+4}^2$ with 2. We still need an additional weight of one for $q_{t}^2$. Hence, this leads to a larger weight.

    \medskip\noindent
    \textbf{Case 1c:} We label a subset of vertices from $q_{t+1}^2,q_{t+2}^2,q_{t+3}^2$ and $q_{t+4}^2$ with 2 and a subset of vertices from $p_{t+1},p_{t+2},p_{t+3}$ and $p_{t+4}$ with 2. It is easy to see that the weight would be at least as large as one of the first two cases.    

    \medskip\noindent
    \textbf{Case 2: }We consider four vertices $q_{t_w}^1$, $q_{t_x}^1$, $q_{t_y}^1$ and $q_{t_z}^1$ such that no two indices among $\{t_w,t_x, t_y, t_z\}$ are consecutive and $b_{t_w}$, $b_{t_x}$, $b_{t_y}$ and $b_{t_z}$ can be covered by a vertex in $A$. If we label all four vertices $q_{t_w}^1$, $q_{t_x}^1$, $q_{t_y}^1$ and $q_{t_z}^1$ with 0, it is easy to see that the weight is strictly large. 

    \medskip\noindent
    \textbf{Case 3: }We consider four vertices $q_{t_w}^1$, $q_{t_x}^1$, $q_{t_y}^1$ and $q_{t_z}^1$ such that a subset of indices among $\{t_w,t_x, t_y, t_z\}$ are consecutive and $b_{t_w}$, $b_{t_x}$, $b_{t_y}$ and $b_{t_z}$ can be covered by a vertex in $A$. If we label all four vertices $q_{t_w}^1$, $q_{t_x}^1$, $q_{t_y}^1$ and $q_{t_z}^1$ with 0, it is easy to see that the weight is strictly large.

    As the weight grows in case of exactly four vertices in $\bigcup\limits_{i \in [4l]} q_i^1$ with a smaller label than 2, the weight also grows for a higher number of vertices of such sort. Therefore, the function $f$ admits a minimum weight if and only if for each vertex $z \in \bigcup\limits_{i \in [4l]} q_i^1$, $f(z) = 2.$
\end{proof}
Lemma \ref{sameindex} also holds for the vertices in the set $\bigcup\limits_{i \in [4l]} q_i^2$. 
\begin{lemma}~\label{onlyone}
    Let $f$ be a minimum weighted Roman dominating function. For each $i \in [4l]$, at most one vertex among $q_i^1$ and $q_i^2$ is assigned the label 2 while the other is assigned the label 0 in $f$.
\end{lemma}
\begin{proof}
    As $q_i^1$ and $q_i^2$ have the same neighbourhood in $P \cup S \cup R$, at most one of them can be labeled 2 while the other should be assigned the label 0. This is true because all the neighbours of $q_i^1$ (or $q_i^2$) are covered by the vertex that is assigned the label 2 among $q_i^1$ and $q_i^2$. The vertex that is assigned the label 0 among $q_i^1$ and $q_i^2$ will be covered by its neighbour in $Q$. Therefore, at most one vertex among $q_i^1$ and $q_i^2$ is assigned the label 2 and the other is assigned the label 0 in $f$.
\end{proof}
For some fixed $i \in [4l]$, if one among $q_i^1$ and $q_i^2$ are to be labeled with 2, for the rest of this section, we assume that $q_i^1$ will be the one to be assigned the label 2 and $q_i^2$ is assigned the label 0.

If there exists a vertex $y \in Q$, such that $f(y) =2$, then from Lemmas \ref{sameindex} and \ref{onlyone}, we obtain that for each vertex $z \in \bigcup\limits_{i \in [4l]} {q_i^1}$, $f(z) = 2$.
\begin{lemma}~\label{nopvertex}
    There exists a minimum weighted Roman dominating function $f$ such that $f(p_i) = 0$ for each $i \in [4l]$.
\end{lemma}
\begin{proof}
Let $f'$ be a minimum weighted Roman dominating function with $f'(p_i) > 0$ for some $i \in [4l]$. We show that this can be adjusted to another minimum weighted Roman dominating function $f$ with $f(p_i) = 0$, thereby proving that there exists a minimum weighted Roman dominating function with $f(p_i) = 0$ for all $i \in [4l]$. We have the following cases based on the value of $f'(p_i)$.

\medskip\noindent
\textbf{Case 1:} $f'(p_i) = 2$.

In this case, $p_i$ covers $\{b_i, q_i^1, q_i^2, s_i^1, s_i^2, s_i^3\}$. We modify $f'(p_i)$ from 2 to 0 in $f$ and $f'(q_i^1)$ from 0 to 2 in $f$. If $i$ is odd, we modify $f'(r_i)$ from 1 to 0 in $f$ and $f'(a_j)$ from 0 to 2 in $f$. $f'$ remains unchanged for all other vertices.

We have that $f$ is a Roman dominating function of weight $f(V) = f'(V)$ and has $f(p_i) = 0$.

\medskip\noindent
\textbf{Case 2:} $f'(p_i) = 1$.

In this case, we can simply set $f(p_i) = 0$, keeping all other labels unchanged. Since $p_i$ remains covered and no uncovered vertices are created, $f'$ remains valid, and 
$f(V) = f'(V) - 1 < f'(V)$,
contradicting the minimality of $f'$. Therefore, assigning $f'(p_i) = 1$ is suboptimal.

In both cases, there exists a minimum weighted Roman dominating function $f$ with $f(p_i) = 0$. Repeating this transformation for every $i \in [4l]$ with $f'(p_i) \neq 0$ yields the desired conclusion.

\end{proof}
\begin{lemma}~\label{weight10l}
    The weight of a Roman dominating function $f$ on $H \in \mathcal{G}$ is at least $10l+1$.
\end{lemma}
\begin{proof}
From Lemma \ref{nopvertex}, we have that there exists a minimum weighted $f$ such that $f(p_i) = 0$ for each $i \in [4l]$. By reconsidering Lemmas \ref{weight1}, \ref{weight8l} and \ref{weight2b} with $p_i = 0$ for each $i \in [4l]$, we obtain that the weight of $f$ on $H \in \mathcal{G}$ is at least $10l+1$. 
\end{proof}
\begin{lemma}~\label{g1}
    For a graph $H \in \mathcal{G}$, X4C problem is a YES-instance if and only if there exists a Roman dominating function of weight exactly $10l+1$.
\end{lemma}
\begin{proof}
        $[\Rightarrow]$ Let $C' \subseteq C$ be a solution to the X4C problem instance. The Roman dominating function $f$ is given as follows.

    \[f(z) = \begin{cases}
0, & \text{if } z \in \left( \{v,w,x\} \cup \bigcup\limits_{C_j \notin C'}a_j \cup B \cup P \cup \bigcup\limits_{i\in[4l]} \{q_i^2\} \cup R \cup S \right), \\
1, & \text{if } z \in \left(\{ u \}\right), \\
2, & \text{if } z \in \left(\bigcup\limits_{C_j \in C'} \{a_j\} \cup \bigcup\limits_{i\in[4l]} \{q_i^1\} \right)
\end{cases}
\]
The weight of $f$ is $10l+1$ and each vertex with the label 0 has a neighbour with the label 2. Thus, we conclude that $f$ is a Roman dominating function of weight $10l+1$.

    [$\Leftarrow$] Let $f$ be a Roman domination function. From Lemma \ref{sameindex}, we have that either no vertex from $Q$ has the label 2 or $4l$ vertices from $\bigcup\limits_{i \in [4l]} q_i^1$ have the label 2. Let us analyze both cases.

\medskip\noindent
    \textbf{Case 1:} $f(y) = 0$, for each vertex $y \in Q$.

    As $f(y) = 0$ for each vertex $y \in Q$, we have that $f(z) = 2$ for each vertex $z \in P$. We also have that $f(z) = 1$ for each vertex $z \in R$. The vertices of $B$ are covered by their neighbours in $P$. The vertices $x$ and $v$ are covered by the vertices of $Q$. The weight utilized on $B \cup P \cup Q \cup R \cup S$ is $10l$. The vertices $\{u, w\} \cup A$ cannot be covered by a remaining weight of one. Hence, we conclude that this case does not lead to a Roman dominating function of weight $10l+1$.  

    \medskip\noindent
    \textbf{Case 2:} $f(y) = 2$, for each vertex $y \in \bigcup\limits_{i \in [4l]} q_i^1$.

    The vertices $\{v,x\} \cup P \cup Q \cup R \cup S$ are covered. The weight utilized on $Q$ is $8l$ and we are left with $2l+1$. We need to label a subset of vertices in $A$ with 2, if not we need to label all the $4l$ vertices in $B$ with 1, which goes beyond the available weight. As $A$ is a clique, the vertex set $\{w\} \cup A$ will be covered by a neighbour from $A$ with label 2. Vertex $u$ is assigned the label 1. The only way to cover the vertex set $B$ with the remaining weight of $2l$ is to set $l$ vertices from $A$ with label 2. If there exist such $l$ vertices that cover all $4l$ vertices of $B$, then we obtain an RDF of weight $10l+1$. If there does not exist such $l$ vertices, then we cannot compute an RDF of weight $10l+1$. The $l$ vertices of $A$ that cover the vertex set $B$ correspond to the sets that cover all the elements in the X4C problem instance exactly once. 
\end{proof}
From Lemmas \ref{weight10l} and \ref{g1}, we have that in order to compute $\gamma_R(H)$ for $H \in \mathcal{G}$, we must solve the X4C problem instance. From Lemma \ref{x4cproof}, we have that X4C problem is NP-complete. Hence, we obtain the following result.
\begin{theorem}
    RD problem on graph class $\mathcal{G}$ is NP-complete.
\end{theorem}
Now, we analyse the complexity of GRD problem on graph class $\mathcal{G}$.
\begin{lemma}~\label{twovertices}
     Let $f$ be a minimum weighted global Roman dominating function. For each $y \in B \cup R \cup S$, $f(y) \neq 2$.
\end{lemma}
 \begin{proof}
    Consider a vertex $s_i^1 \in S$. By having the label 2 for $s_i^1$, the vertices $s_i^2$ and $s_i^3$ are not covered in $G$. All the vertices $s_i^1$ covers in $G'$ are covered by some other vertices with the label 2. Therefore, the weight of $f$ increases by exactly two.
    
    Consider a vertex $r_i \in R$. Even if we have the label 2 for $r_i$, the label 2 must be assigned to one among $p_i, q_i^1$ and $q_i^2$. All the vertices $r_i$ covers in $G'$ are covered by some other vertices with the label 2. Hence, this leads to an increased weight of at least one.
    
    Consider a vertex $b_i \in B$. Each of its neighbours from $A$ is covered either by a vertex from $\{v,w,x\} $ or by some vertex of $A$. Vertex $p_i$ is also covered either by $p_i$ itself or one among $q_i^1$ and $q_i^2$. All the vertices $b_i$ covers in $G'$ are covered by some other vertices with the label 2. Therefore, assigning the label 2 to $b_i$ increases the weight of $f$ by at least one.
 \end{proof}
\begin{lemma}~\label{weight1grd}
     The minimum weight of a global Roman dominating function $f$ on $\{u,v,w,x\}$ is at least $2$.
\end{lemma}
\begin{proof}
    From Lemma \ref{twovertices}, we have that the vertices $v$ and $x$ do not have a non-neighbour with the label 2. Therefore, one among these two vertices must be assigned the label 2 so that the other vertex can have a non-neighbour with the label 2. The vertex that is assigned the label 2 among $v$ and $x$ will also cover $u$. 
\end{proof}
\begin{lemma}~\label{weight10l2}
    The weight of a global Roman dominating function $f$ on $H \in \mathcal{G}$ is at least $10l+2$.
\end{lemma}
\begin{proof}
    It is to be noted that Lemmas \ref{weight8l}, \ref{weight2b}, \ref{sameindex}, \ref{onlyone} and \ref{nopvertex} also hold for global Roman dominating function. From Lemma \ref{nopvertex}, we have that $f(p_i) = 0$ for each $i \in [4l]$. By reconsidering Lemmas \ref{weight8l}, \ref{weight2b} and \ref{weight1grd} with $p_i = 0$ for each $i \in [4l]$, we have that the weight of a global Roman dominating function $f$ on $H \in \mathcal{G}$ is at least $10l+2$. 
\end{proof}
\begin{lemma}~\label{grd-conclude}
    For graph class $\mathcal{G}$, there exists a minimum weighted global Roman dominating function of weight exactly $10l+2$.
\end{lemma}
\begin{proof}
    Let X4C problem be a YES-instance:

 The minimum weighted Roman dominating function $f$ is given as follows.
\[f(z) = \begin{cases}
0, & \text{if } z \in \left( \{u, v,x\} \cup \bigcup\limits_{C_j \notin C'}\{a_j\} \cup B \cup P \cup \bigcup\limits_{i\in[4l]} \{q_i^2\} \cup R \cup S \right), \\
1, & \text{if } z \in \emptyset, \\
2, & \text{if } z \in \left(\{w\} \cup \bigcup\limits_{C_j \in C'} \{a_j\} \cup \bigcup\limits_{i\in[4l]} \{q_i^1\} \right)
\end{cases}
\]
The weight of $f$ is $10l+2$. Each vertex with label 0 does have a neighbour and also a non-neighbour with the label 2. Thus, we conclude that $f$ is a global Roman dominating function of weight $10l+2$.

\medskip
\noindent 
Let X4C problem be a NO-instance:

 The minimum weighted Roman dominating function $f$ is given as follows.
    \[f(z) = \begin{cases}
0, & \text{if } z \in \left( \{u, v,x\} \cup A \cup B \cup Q \cup S \right), \\
1, & \text{if } z \in \emptyset, \\
2, & \text{if } z \in \left(\{w\} \cup P \cup R \right)
\end{cases}
\]

The weight of $f$ is $10l+2$. Each vertex with label 0 has a neighbour and also a non-neighbour with label 2. Thus, we conclude that $f$ is a global Roman dominating function of weight $10l+2$.
\end{proof}
From Lemmas \ref{weight10l2} and \ref{grd-conclude}, we have that $\gamma_{gR}(G) = 10l+2$. Hence, we obtain the following theorem.
\begin{theorem}
    GRD problem on graph class $\mathcal{G}$ is linear-time solvable.
\end{theorem}
\section{Restricted Graph Classes}
In this section, we investigate the complexity of GRD problem and prove that the problem is NP-complete on split graphs, chordal bipartite graphs, planar bipartite graphs with maximum degree five and circle graphs. Furthermore, we present a linear-time algorithm for cographs.
\subsection{Split graphs}
In this subsection, we prove that GRD problem is NP-complete on split graphs following a reduction from \xthc{} (X3C).

 The definition of split graphs is given as follows.
\begin{definition}
    Graph $G = (V, E)$ is a split graph if the vertex set $V$ can be partitioned into two sets $V_1$ and $V_2$ such that the graph induced on $V_1$ is a clique and $V_2$ is an independent set in $G$.
\end{definition}
  The complexity result related to X3C from the literature is given as follows.
\begin{theorem}[\cite{johnson1979computers}]~\label{knownresult1}
    \xthc{} is NP-complete.
\end{theorem}
 \noindent \textbf{Construction. }Let $(X = \{x_1, x_2, \ldots, x_{3q}\}, C = \{C_1, C_2, \ldots, C_t\})$ be an instance of X3C, we transform it into an instance $I = (G, k)$ of GRD problem as follows. For each set $C_j\in C$, we create a vertex $a_j$ and for each element $x_i \in X$, we create a vertex $b_i$ in $G$. We also introduce a dummy vertex $a_{t+1}$. Let $A = \bigcup\limits_{j \in [t+1]}a_j$ and $B = \bigcup\limits_{i\in [3q]}b_i$. The vertices of set $A$ form a clique. For each vertex in $a_j \in A$ we also add a pendant vertex. The pendant vertex adjacent to $a_j$ is denoted by $p_j$. Let $P = \bigcup\limits_{j\in [t+1]}p_j$. We add an edge between $a_j$ and $b_i$ if and only if $x_i \in C_j$ in the X3C instance. This concludes the construction of $G$. See Figure~\ref{fig:fig3} for an illustration. We set $k$ = $q+t+3$. Graph $G$ is a split graph with $A$ being the clique and $B \cup P$ being an independent set in $G$.

\begin{figure} [t]
    \centering
     \begin{tikzpicture} [thick,scale=0.75, every node/.style={scale=0.9}]

        \filldraw (5, 5) circle (2pt) node[anchor=south]{};
        \filldraw (7, 5) circle (2pt) node[anchor=south]{};
        \filldraw (9, 5) circle (2pt) node[anchor=south]{};
        \filldraw (11, 5) circle (2pt) node[anchor=south]{};
        \filldraw (13, 5) circle (2pt) node[anchor=south]{};

        \filldraw (5, 2.5) circle (2pt) node[anchor=south]{};
        \filldraw (7, 2.5) circle (2pt) node[anchor=south]{};
        \filldraw (9, 2.5) circle (2pt) node[anchor=south]{};
        \filldraw (11, 2.5) circle (2pt) node[anchor=south]{};
        \filldraw (13, 2.5) circle (2pt) node[anchor=south]{};

        \filldraw (5, 8) circle (2pt) node[anchor=south]{};
        \filldraw (7, 8) circle (2pt) node[anchor=south]{};
        \filldraw (9, 8) circle (2pt) node[anchor=south]{};
        \filldraw (11, 8) circle (2pt) node[anchor=south]{};
        \filldraw (3, 8) circle (2pt) node[anchor=south]{};
        \filldraw (13, 8) circle (2pt) node[anchor=south]{};
        
        \draw[thin] (5, 5) -- (3, 8);
        \draw[thin] (5, 5) -- (5, 8);
        \draw[thin] (5, 5) -- (9, 8);
        \draw[thin] (7, 5) -- (5, 8);
        \draw[thin] (7, 5) -- (7, 8);
        \draw[thin] (7, 5) -- (9, 8);        
        \draw[thin] (9, 5) -- (7, 8);
        \draw[thin] (9, 5) -- (11, 8);
        \draw[thin] (9, 5) -- (13, 8);
        \draw[thin] (11, 5) -- (9, 8);
        \draw[thin] (11, 5) -- (11, 8);
        \draw[thin] (11, 5) -- (13, 8);  

        \draw[thin] (5, 5) -- (5, 2.5);
        \draw[thin] (7, 5) -- (7, 2.5);
        \draw[thin] (9, 5) -- (9, 2.5);
        \draw[thin] (11, 5) -- (11, 2.5);
        \draw[thin] (13, 5) -- (13, 2.5);

        \draw[thick, dotted] (2, 7.5) -- (2, 8.5);
        \draw[thick, dotted] (14, 7.5) -- (14, 8.5);
        \draw[thick, dotted] (2, 7.5) -- (14, 7.5);
        \draw[thick, dotted] (2, 8.5) -- (14, 8.5);

        \draw[thick, dotted] (4, 4.5) -- (4, 5.5);
        \draw[thick, dotted] (14, 4.5) -- (14, 5.5);
        \draw[thick, dotted] (4, 4.5) -- (14, 4.5);
        \draw[thick, dotted] (4, 5.5) -- (14, 5.5);

        \draw[thick, dotted] (4, 2) -- (4, 3);
        \draw[thick, dotted] (14, 2) -- (14, 3);
        \draw[thick, dotted] (4, 2) -- (14, 2);
        \draw[thick, dotted] (4, 3) -- (14, 3);

        \filldraw (1, 7.5) circle (0cm) node[anchor=south]{$B$};
        \filldraw (3, 4.5) circle (0cm) node[anchor=south]{$A$};

        \filldraw (2.65, 7.6) circle (0cm) node[anchor=south]{$b_1$};
        \filldraw (4.65, 7.6) circle (0cm) node[anchor=south]{$b_2$};
        \filldraw (6.65, 7.6) circle (0cm) node[anchor=south]{$b_3$};
        \filldraw (9.4, 7.6) circle (0cm) node[anchor=south]{$b_4$};
        \filldraw (11.4, 7.6) circle (0cm) node[anchor=south]{$b_5$};
        \filldraw (13.4, 7.6) circle (0cm) node[anchor=south]{$b_6$};
        
        \filldraw (4.65, 4.6) circle (0cm) node[anchor=south]{$a_1$};
        \filldraw (6.65, 4.6) circle (0cm) node[anchor=south]{$a_2$};
        \filldraw (8.65, 4.6) circle (0cm) node[anchor=south]{$a_3$};
        \filldraw (10.65, 4.6) circle (0cm) node[anchor=south]{$a_4$};
        \filldraw (12.65, 4.6) circle (0cm) node[anchor=south]{$a_5$};

        \filldraw (4.65, 2.1) circle (0cm) node[anchor=south]{$p_1$};
        \filldraw (6.65, 2.1) circle (0cm) node[anchor=south]{$p_2$};
        \filldraw (8.65, 2.1) circle (0cm) node[anchor=south]{$p_3$};
        \filldraw (10.65, 2.1) circle (0cm) node[anchor=south]{$p_4$};
        \filldraw (12.65, 2.1) circle (0cm) node[anchor=south]{$p_5$};
        
        \filldraw (3, 2) circle (0cm) node[anchor=south]{$P$};

        \filldraw (9, 0.5) circle (0cm) node[anchor=south]{$G$};
    \end{tikzpicture}
        \caption{Reduced instance of GRD problem constructed from an X3C instance: $X = \{1,2,3,4,5,6\}$ and $C$ = $\{C_1 =  \{1,2,4\}$, $C_2 =  \{2,3,4\}$, $C_3 =  \{3,5,6\}$ and $C_4 =  \{4,5,6\}\}$. The vertices of $A$ form a clique, which is not shown in the figure.}
    \label{fig:fig3}
\end{figure}
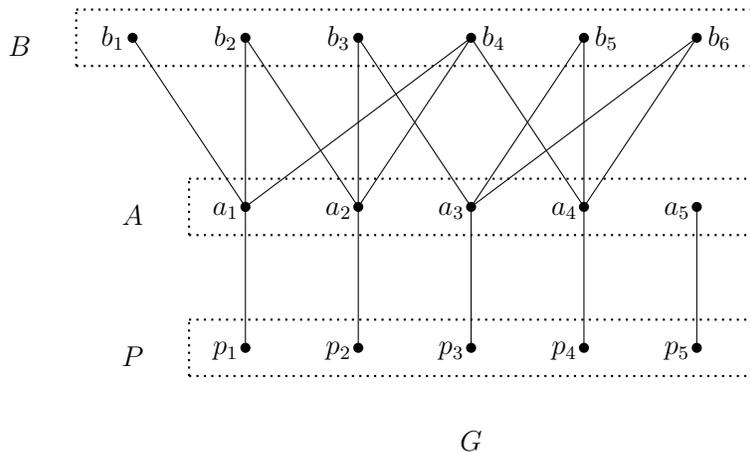
\begin{lemma} ~\label{split1}
    $(X, C)$ is a YES-instance if and only if $I$ has a global Roman dominating function of weight $q+t+3$.
\end{lemma}
\begin{proof}
    ($\Rightarrow$) Let $C'\subseteq C$ be a solution to X3C instance. We define a global Roman dominating function $f$ as follows.
    \[
f(z) = 
\begin{cases}
0, & \text{if } z \in \left(\bigcup\limits_{C_j \notin C'} \{a_j\} \cup \bigcup\limits_{C_j \in C'} \{p_j\} \cup B\right), \\
1, & \text{if } z \in \left(\{a_{t+1}\} \cup \bigcup\limits_{C_j \notin C'} \{p_j\}\right), \\
2, & \text{if } z \in \left(\{p_{t+1}\} \cup \bigcup\limits_{C_j \in C'} \{a_j\}\right)
\end{cases}
\]
Each vertex $u \in B$ has a neighbour in $A$ with the label 2 and also has a non-neighbour $p_{t+1}$ with the label 2. Each vertex $u \in A$ with $f(u)= 0 $ has a neighbour from $A$ with the label 2 and a non-neighbour $p_{t+1}$ with the label 2. Consider a vertex $u \in P$ with $f(u) = 0$. The only neighbour of $u$ in $A$ has the label 2 and $u$ has a non-neighbour $p_{t+1}$ with the label 2. Hence, we conclude that $f$ is a global Roman dominating function for $G$ of weight $q+t+3$. \vspace{2mm} \\
    ($\Leftarrow$) Let $f$ be a global Roman domination function for $G$ of weight $q+t+3$. For each pendant vertex $p_j \in P$, we have that either $f(p_j) = 1$ or $f(a_j) = 2$. Similarly, for each vertex $b_i \in B$, either $f(b_i) = 1$ or $f(a_j) = 2$ for some vertex $a_j \in N(b_i)$. Hence, the weight of $A \cup B$ is at least $2q$ and the weight of $A \cup P$ is at least $t+1$. 
    
    Now, we argue that, $f$ is minimum weighted if and only if each vertex from $B$ has the label 0 and exactly $q$ vertices from $A$ are assigned the label 2. For the sake of contradiction, let us assume that there exist exactly three vertices from $B$ with the label 1 and $q-1$ vertices from $A$ with the label 2. Among the remaining $t-q+2$ uncovered pendant vertices of $P$, $t-q+1$ are assigned the label 1 and one pendant vertex (say $p_j$) is assigned the label 2. Vertex $p_j$ is labeled 2 to serve as a non-neighbour with the label 2 for all the vertices with the label 0, except for $a_j$. As $p_j$ has the label 2, $a_j$ must be labeled 1. The weight of $f$ becomes $q+t+5$, which is higher than the value of $k$. It is easy to see that if we label more than three vertices of $B$ with 1, we would still need a total weight of more than $q+t+5$. Hence, we must label each vertex of $B$ with 0. Therefore, exactly $q$ vertices from $A$ are assigned the label 2. 
    
    Among the $t-q+1$ uncovered pendant vertices of $P$, $t-q$ are assigned the label 1 and one pendant vertex (say $p_j$) is assigned the label 2. As $p_j$ has the label 2, $a_j$ must be labeled 1. The weight of $f$ becomes $q+t+3$. The vertices from $\{a_1, a_{2}, \ldots, a_t\}$ that are assigned the label 2 will correspond to $C'$ in the X3C instance. 
\end{proof}
As $G$ is a split graph, with the help of Theorem \ref{knownresult1} and Lemma \ref{split1}, we arrive at the following theorem.
\begin{theorem}
    GRD problem on split graphs is NP-complete.
\end{theorem}
\subsection{Chordal bipartite graphs, planar bipartite graphs with maximum degree five and circle graphs}

We prove that GRD problem is NP-complete on chordal bipartite graphs, planar bipartite graphs with maximum degree five and circle graphs following a reduction from DS problem.

 The definition of chordal bipartite graphs and circle graphs is given as follows.
\begin{definition}
    A bipartite graph is chordal bipartite if each cycle of length at least six has a chord.
\end{definition}
 \begin{definition}
    A graph is a circle graph if and only if there exists a corresponding circle model in which the chords represent the vertices of the graph and two chords intersect only if the corresponding vertices are adjacent.
\end{definition}

The known complexity result related to DS problem is given as follows.
\begin{theorem}[\cite{muller1987np,zvervich1995induced,keil1993complexity}]~\label{knownresult2}
    DS problem on chordal bipartite graphs, planar bipartite graphs with maximum degree three and circle graphs is NP-complete. 
\end{theorem}
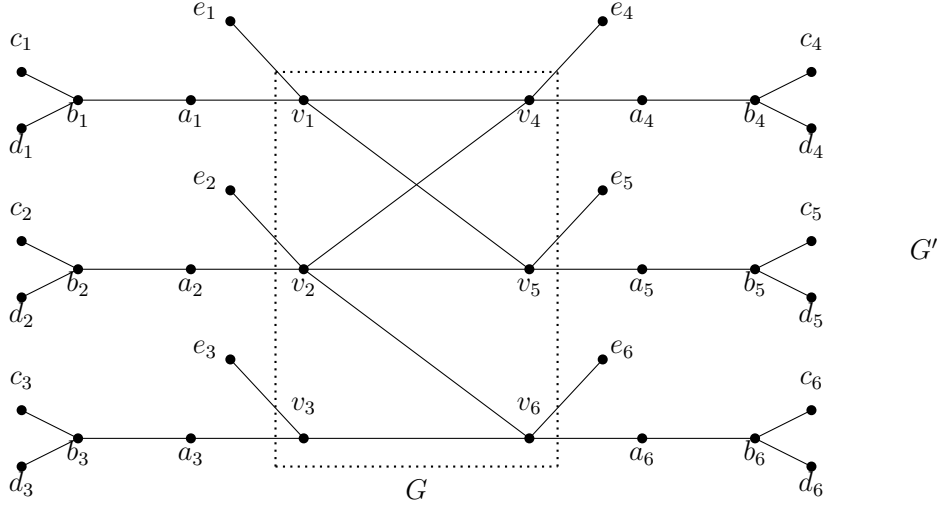
\begin{figure}
    \centering
     \begin{tikzpicture} [thick,scale=0.75, every node/.style={scale=0.9}]

        \filldraw (5, 2) circle (2pt) node[anchor=south]{};
        \filldraw (5, 5) circle (2pt) node[anchor=south]{};
        \filldraw (5, 8) circle (2pt) node[anchor=south]{};
        \filldraw (9, 2) circle (2pt) node[anchor=south]{};
        \filldraw (9, 5) circle (2pt) node[anchor=south]{};
        \filldraw (9, 8) circle (2pt) node[anchor=south]{};

        \filldraw (3.7, 3.4) circle (2pt) node[anchor=south]{};
        \filldraw (3.7, 6.4) circle (2pt) node[anchor=south]{};
        \filldraw (3.7, 9.4) circle (2pt) node[anchor=south]{};
        \filldraw (10.3, 3.4) circle (2pt) node[anchor=south]{};
        \filldraw (10.3, 6.4) circle (2pt) node[anchor=south]{};
        \filldraw (10.3, 9.4) circle (2pt) node[anchor=south]{};

        \draw[thin] (5, 2) -- (3.7, 3.4);
        \draw[thin] (5, 5) -- (3.7, 6.4);
        \draw[thin] (5, 8) -- (3.7, 9.4);
        \draw[thin] (9, 2) -- (10.3, 3.4);
        \draw[thin] (9, 5) -- (10.3, 6.4);
        \draw[thin] (9, 8) -- (10.3, 9.4);  

        \draw[thick, dotted] (4.5, 1.5) -- (9.5, 1.5);
        \draw[thick, dotted] (4.5, 8.5) -- (9.5, 8.5);
        \draw[thick, dotted] (4.5, 1.5) -- (4.5, 8.5);
        \draw[thick, dotted] (9.5, 1.5) -- (9.5, 8.5);

        \draw[thin] (5, 2) -- (9, 2);
        \draw[thin] (5, 5) -- (9, 5);
        \draw[thin] (5, 8) -- (9, 8);
        \draw[thin] (5, 8) -- (9, 5);
        \draw[thin] (5, 5) -- (9, 2);
        \draw[thin] (5, 5) -- (9, 8);   
        
        \draw[thin] (5, 2) -- (3, 2);
        \draw[thin] (5, 5) -- (3, 5);
        \draw[thin] (5, 8) -- (3, 8);
        \draw[thin] (3, 2) -- (1, 2);
        \draw[thin] (3, 5) -- (1, 5);
        \draw[thin] (3, 8) -- (1, 8);        
        \draw[thin] (9, 2) -- (11, 2);
        \draw[thin] (9, 5) -- (11, 5);
        \draw[thin] (9, 8) -- (11, 8);
        \draw[thin] (11, 2) -- (13, 2);
        \draw[thin] (11, 5) -- (13, 5);
        \draw[thin] (11, 8) -- (13, 8);

        \draw[thin] (0, 1.5) -- (1, 2);
        \draw[thin] (0, 2.5) -- (1, 2);
        \draw[thin] (0, 4.5) -- (1, 5);
        \draw[thin] (0, 5.5) -- (1, 5);
        \draw[thin] (0, 7.5) -- (1, 8);
        \draw[thin] (0, 8.5) -- (1, 8);
        \draw[thin] (14, 7.5) -- (13, 8);
        \draw[thin] (14, 8.5) -- (13, 8);
        \draw[thin] (14, 4.5) -- (13, 5);
        \draw[thin] (14, 5.5) -- (13, 5);
        \draw[thin] (14, 1.5) -- (13, 2);
        \draw[thin] (14, 2.5) -- (13, 2);

        \filldraw (3, 2) circle (2pt) node[anchor=south]{};
        \filldraw (3, 5) circle (2pt) node[anchor=south]{};
        \filldraw (3, 8) circle (2pt) node[anchor=south]{};
        \filldraw (11, 2) circle (2pt) node[anchor=south]{};
        \filldraw (11, 5) circle (2pt) node[anchor=south]{};
        \filldraw (11, 8) circle (2pt) node[anchor=south]{};        
        \filldraw (1, 2) circle (2pt) node[anchor=south]{};
        \filldraw (1, 5) circle (2pt) node[anchor=south]{};
        \filldraw (1, 8) circle (2pt) node[anchor=south]{};
        \filldraw (13, 2) circle (2pt) node[anchor=south]{};
        \filldraw (13, 5) circle (2pt) node[anchor=south]{};
        \filldraw (13, 8) circle (2pt) node[anchor=south]{};
     
        \filldraw (0, 1.5) circle (2pt) node[anchor=south]{};        
        \filldraw (0, 2.5) circle (2pt) node[anchor=south]{};
        \filldraw (0, 4.5) circle (2pt) node[anchor=south]{};        
        \filldraw (0, 5.5) circle (2pt) node[anchor=south]{};        
        \filldraw (0, 7.5) circle (2pt) node[anchor=south]{};        
        \filldraw (0, 8.5) circle (2pt) node[anchor=south]{};

        \filldraw (14, 1.5) circle (2pt) node[anchor=south]{};        
        \filldraw (14, 2.5) circle (2pt) node[anchor=south]{};
        \filldraw (14, 4.5) circle (2pt) node[anchor=south]{};        
        \filldraw (14, 5.5) circle (2pt) node[anchor=south]{};        
        \filldraw (14, 7.5) circle (2pt) node[anchor=south]{};        
        \filldraw (14, 8.5) circle (2pt) node[anchor=south]{};

        \filldraw (7, 0.75) circle (0cm) node[anchor=south]{$G$};
        \filldraw (16, 5) circle (0cm) node[anchor=south]{$G'$};

        \filldraw (5, 7.35) circle (0cm) node[anchor=south]{$v_1$};
        \filldraw (3, 7.35) circle (0cm) node[anchor=south]{$a_1$};
        \filldraw (1, 7.35) circle (0cm) node[anchor=south]{$b_1$};
        \filldraw (0, 8.7) circle (0cm) node[anchor=south]{$c_1$};
        \filldraw (0, 6.8) circle (0cm) node[anchor=south]{$d_1$};
        \filldraw (3.25, 9.25) circle (0cm) node[anchor=south]{$e_1$}; 

        \filldraw (5, 4.35) circle (0cm) node[anchor=south]{$v_2$};
        \filldraw (3, 4.35) circle (0cm) node[anchor=south]{$a_2$};
        \filldraw (1, 4.35) circle (0cm) node[anchor=south]{$b_2$};
        \filldraw (0, 5.7) circle (0cm) node[anchor=south]{$c_2$};
        \filldraw (0, 3.8) circle (0cm) node[anchor=south]{$d_2$};
        \filldraw (3.25, 6.25) circle (0cm) node[anchor=south]{$e_2$};

        \filldraw (5, 2.25) circle (0cm) node[anchor=south]{$v_3$};
        \filldraw (3, 1.35) circle (0cm) node[anchor=south]{$a_3$};
        \filldraw (1, 1.35) circle (0cm) node[anchor=south]{$b_3$};
        \filldraw (0, 2.7) circle (0cm) node[anchor=south]{$c_3$};
        \filldraw (0, 0.8) circle (0cm) node[anchor=south]{$d_3$};
        \filldraw (3.25, 3.25) circle (0cm) node[anchor=south]{$e_3$};

        \filldraw (9, 7.35) circle (0cm) node[anchor=south]{$v_4$};
        \filldraw (11, 7.35) circle (0cm) node[anchor=south]{$a_4$};
        \filldraw (13, 7.35) circle (0cm) node[anchor=south]{$b_4$};
        \filldraw (14, 8.7) circle (0cm) node[anchor=south]{$c_4$};
        \filldraw (14, 6.8) circle (0cm) node[anchor=south]{$d_4$};
        \filldraw (10.65, 9.25) circle (0cm) node[anchor=south]{$e_4$}; 

        \filldraw (9, 4.35) circle (0cm) node[anchor=south]{$v_5$};
        \filldraw (11, 4.35) circle (0cm) node[anchor=south]{$a_5$};
        \filldraw (13, 4.35) circle (0cm) node[anchor=south]{$b_5$};
        \filldraw (14, 5.7) circle (0cm) node[anchor=south]{$c_5$};
        \filldraw (14, 3.8) circle (0cm) node[anchor=south]{$d_5$};
        \filldraw (10.65, 6.25) circle (0cm) node[anchor=south]{$e_5$};

        \filldraw (9, 2.25) circle (0cm) node[anchor=south]{$v_6$};
        \filldraw (11, 1.35) circle (0cm) node[anchor=south]{$a_6$};
        \filldraw (13, 1.35) circle (0cm) node[anchor=south]{$b_6$};
        \filldraw (14, 2.7) circle (0cm) node[anchor=south]{$c_6$};
        \filldraw (14, 0.8) circle (0cm) node[anchor=south]{$d_6$};
        \filldraw (10.65, 3.25) circle (0cm) node[anchor=south]{$e_6$};
        
    \end{tikzpicture}
        \caption{Reduced instance of GRD problem from DS problem instance.}
    \label{fig:fig4}
\end{figure}

\medskip
 \noindent \textbf{Construction. }Let $I = (G,k)$ be an instance of DS problem. We transform it into an instance $I' = (G', 3n+k)$ of GRD problem as follows. For each vertex $v_i \in V$, we attach a tree rooted at $v_i$ as follows. We create two vertices $a_i$ and $e_i$ and make them adjacent to $v_i$. We introduce a vertex $b_i$ and make it adjacent to $a_i$. We create two pendant vertices $c_i$ and $d_i$ and make them adjacent to $b_i$. This concludes the construction of $G'$. See Figure~\ref{fig:fig4} for more details.
\begin{lemma}~\label{chordal1}
    $G$ has a dominating set of size at most $k$ if and only if $G'$ has a global Roman dominating function of weight at most $3n+k$.
\end{lemma}
\begin{proof}
        ($\Rightarrow$) Let $S$ be a dominating set of size at most $k$. We define a global Roman dominating function $f$ as follows. 
\[
f(z) = 
\begin{cases}
0, & \text{if } z \in \left(\bigcup\limits_{i \in [n]}\{a_i, c_i, d_i\} \cup \bigcup\limits_{v_i \in S} e_i \cup (V\setminus S)\right), \\
1, & \text{if } z \in \left(\bigcup\limits_{v_i \notin S} \{e_i\} \right), \\
2, & \text{if } z \in \left(S \cup \bigcup\limits_{i \in [n]} \{b_i\}\right)
\end{cases}
\]
    Each vertex $u \in V'$ with $f(u) = 0$ has a non-neighbour $v \notin N(v)$ with $f(v) =2$ in $G'$. This is true because, if we consider any two vertices $b_i$ and $b_j$, each vertex in $u \in V'$ is either non-adjacent to $b_i$ or non-adjacent to $b_j$. Hence, we only discuss whether each vertex $u \in V'$ with $f(u) = 0$ has a neighbour $v \in N(u)$ such that $f(v) =2$.

    \begin{itemize}
        \item Each vertex in $\bigcup\limits_{i \in [n]}\{a_i, c_i,d_i\}$ has a neighbour $b_i$ such that $f(b_i) = 2$.
        \item Each vertex in $\bigcup\limits_{v_i \in S} e_i$ has a neighbour $v_i$ such that $f(v_i) = 2$.
        \item Each vertex $u \in V\setminus S$ has a neighbour $v \in N(u)$ such that $f(v) = 2$.
    \end{itemize}
    Hence, we conclude that $f$ is a global Roman dominating function with weight at most $3n+k$. \vspace{2mm} \\
($\Leftarrow$) Let $f$ be a global Roman dominating function of weight at most $3n+k$ and $S$ be a dominating set.
\begin{itemize}
    \item Each vertex in $\bigcup\limits_{i \in [n]} b_i$ is assigned the label two and its three neighbours $a_i$, $c_i$ and $d_i$ are assigned the label 0.
    \item At this point, we are left with a weight of $n+k$. Each vertex in  $\bigcup\limits_{i \in [n]}\{v_i, e_i\}$ is yet to be covered.
    \item Vertex pair $(e_i, v_i)$ must be assigned labels either (1, 0) or (0, 2). For at most $k$ of such pairs, we label $v_i$ with 2 and $e_i$ with 0. For the rest of at most $n-k$ such pairs, we label $v_i$ with 0 and $e_i$ with 1. The weight of $f$ is $2k+1(n-k)$, that is, $n+k$. 
    \item The vertex set $V' \subseteq V$ that is assigned the label two in $f$ will correspond to the dominating set $S$ in $G$ of size at most $k$.
\end{itemize}  
\end{proof}
\begin{lemma}~\label{graphclass-preserv} 
    If $G$ is
    \begin{enumerate}
        \item chordal bipartite then $G'$ is chordal bipartite.
        \item planar bipartite with maximum degree three, then $G'$ is planar bipartite with maximum degree five.
        \item a circle graph, then $G'$ is also a circle graph.
    \end{enumerate}
\end{lemma}
\begin{proof}
    \begin{enumerate} 
    \medskip
    \item \textbf{Chordal bipartite.}  
    In the construction of $G'$, each gadget added for a vertex $v_i \in V$ consists of a tree involving the vertices $a_i, b_i, c_i, d_i$ and $e_i$. As no cycle is formed within the vertices of a gadget, no cycles of length greater than 4 are introduced.
    As $G$ is chordal bipartite, $G'$ remains chordal bipartite.

    \item \textbf{Planar bipartite with maximum degree 5.} 
    Each gadget is a tree with maximum degree 3. The vertex $v_i$ in $G'$ has a degree of 5 after the addition of edges to $a_i$ and $e_i$. The planarity of $G'$ is preserved. As the gadget for any vertex $v_i \in G'$ does not introduce any cycle, the bipartiteness of $G'$ is preserved. Therefore, $G'$ is a planar bipartite graph with maximum degree 5.

    \item \textbf{Circle graph.}  
    In the construction, each gadget for vertex $v_i$ is confined only to $v_i$. The gadget forms a tree rooted at $v_i$, and does not interact with other vertices of $G$ and their respective gadgets.  
    In the circle representation, the gadget can be represented by adding chords corresponding to $a_i, b_i, c_i, d_i, e_i$ near the chord for $v_i$ such that these chords do not intersect with any other chords. Hence, the resulting graph $G'$ is also a circle graph.
\end{enumerate}
\end{proof}
\begin{theorem}~\label{cbg}
    GRD problem is NP-complete on chordal bipartite graphs, planar bipartite graphs with maximum degree five and circle graphs.
\end{theorem}
\subsection{Cographs}
In this subsection, we present a linear-time algorithm for computing the global Roman domination number of a cograph.

The definition of cographs is given as follows.
\begin{definition}
    A cograph is a graph that does not contain any induced path on four vertices.
\end{definition}
If $G$ is a connected cograph then $\gamma_R(G) \leq 4$~\cite{liedloff2008efficient}. However, $\gamma_{gR}(G)$ not necessarily  bounded. For example, the complete graph $K_n$ is a cograph and $\gamma_{gR}(K_n)=n$. 

\begin{lemma}[\cite{corneil1985linear}]~\label{cograph-structure}
    If $G$ is a cograph then one of the following holds.
    \begin{enumerate}
        \item $G$ has at most one vertex.
        \item $G$ is the disjoint union of two cographs $G_1$ and $G_2$, i.e.,
        $G=G_1 \cup G_2$.
        \item $G$ is the join of two cographs $G_1$ and $G_2$, i.e.,
        $G=G_1 \vee G_2$.
    \end{enumerate}
\end{lemma}
A cotree $T_G$ of a cograph $G$ is a rooted tree in which each internal vertex is either of
$\cup$ (union) type  or $\vee$ (join) type. The leaves of $T_G$ correspond to the vertices of the cograph $G$.

\begin{theorem}\cite{liedloff2008efficient}
If $G$ is a cograph then $\gamma_R(G)$ can be computed in linear-time.
\end{theorem}

\begin{theorem}\cite{atapour2015global} \label{th-disconnected}
    If $G$ is a disconnected graph with at least two components of order at
least two, then $\gamma_{gR}(G)=\gamma_{R}(G)$.
\end{theorem}

\begin{lemma}\label{cograph-union}
    Let $G_1$ and $G_2$ be two cographs with each of them having a component of at least two vertices and  $G = G_1 \cup G_2$. Then 
    \begin{enumerate}
        \item $\gamma_{gR}(G)=\gamma_{R}(G)=\gamma_{R}(G_1)+\gamma_{R}(G_2)$
        \item $\gamma_{R}(\overline G)=\gamma_{R}(\overline G_1 \vee \overline G_2)$
    \end{enumerate}
\end{lemma}
\begin{proof}
    Since $G$ is disconnected and consists of two components, $G_1$ and $G_2$, each having at least two vertices. It follows from Theorem~\ref{th-disconnected} that $\gamma_{gR}(G)=\gamma_{R}(G)$. If $G$ contains two components $G_1$ and $G_2$ such that each of them has a component of at least two vertices then it is easy to see that $\gamma_{R}(G)=\gamma_{R}(G_1)+\gamma_{R}(G_2)$.

    The second part of the lemma follows from the fact that the complement of $G$ is the join of the complements of the individual graphs $G_1$ and $G_2$. 
\end{proof}

If either $G_1$ or $G_2$ (without loss of generality, $G_1$) has no component with at least two vertices, we examine whether $G_2$ contains exactly one component with at least three vertices and its remaining components are isolated vertices. If $G_2$ is of this form, then $\gamma_{gR}(G)=\gamma_{R}(G)+1$, otherwise $\gamma_{gR}(G)=\gamma_{R}(G)$. If neither $G_1$ nor $G_2$ has a component with at least two vertices, then $\gamma_{gR}(G)=\gamma_{R}(G)$.

\begin{lemma}\label{cograph-join}
    Let $G_1$ and $G_2$ be two cographs each with at least two vertices and  $G = G_1 \vee G_2$. Then 
    \begin{enumerate}
        \item $\gamma_{gR}(G)=\gamma_{R}(\overline{G_1} )+\gamma_{R}(\overline{G_2})$
        \item $\gamma_{R}(\overline G)=\gamma_{R}(\overline{G_1} )+\gamma_{R}(\overline{G_2})$
    \end{enumerate}
\end{lemma}
\begin{proof}
    We know that $\gamma_{gR}(G)=\gamma_{gR}(\overline G)$, where $\overline G$ is the disjoint union of two graphs $\overline{G_1}$ and $\overline{G_2}$. Hence, by  Lemma~\ref{cograph-union}, we obtain $$\gamma_{gR}(G)=\gamma_{gR}(\overline G)=\gamma_{R}(\overline{G})=\gamma_{R}(\overline{G_1} )+\gamma_{R}(\overline{G_2})$$

    Since $\overline{G}$ is disconnected and consists of the two components $\overline{G_1}$ and $\overline{G_2}$, it follows that $\gamma_{R}(\overline G)=\gamma_{R}(\overline{G_1} )+\gamma_{R}(\overline{G_2})$. 
\end{proof}
Note that in Lemmas~\ref{cograph-union} and~\ref{cograph-join}, if at least one of the graphs $G_1$ and $G_2$ consists of exactly one vertex, then $G$ either contains a universal vertex or is disconnected. In such cases, it is straightforward to compute $\gamma_{R}(G)$, $\gamma_{R}(\overline{G})$, $\gamma_{gR}(G)$ and $\gamma_{gR}(\overline{G})$.

From Lemmas~\ref{cograph-union} and~\ref{cograph-join}, we have the following theorem.
\begin{theorem}
If $G$ is a cograph then $\gamma_{gR}(G)$ can be computed in linear-time.
\end{theorem}
\section{Conclusion}
In this paper, we studied the complexity of GRD problem and obtained several interesting results. We provided graph classes to differentiate the complexity of GRD problem from that of RD problem. We proved that GRD problem is NP-complete on split graphs, chordal bipartite graphs, planar bipartite graphs with maximum degree five and circle graphs. On the positive front, we presented a linear-time algorithm for GRD problem on cographs. 
\medskip

\noindent
We conclude the paper with the following open questions on GRD problem. 
\begin{itemize}
    \item What is the complexity of the problem on interval graphs, block graphs and strongly chordal graphs?
    \item Since the problem is known to be W[2]-hard parameterized by weight, does it remain W[2]-hard on split graphs and bipartite graphs?
\end{itemize}



\bibliographystyle{elsarticle-num}
\bibliography{references}
\end{document}